\documentclass{alggeom}
\usepackage[mathscr]{eucal}
\usepackage{amsmath, palatino, mathpazo, amsfonts, mathrsfs}
\usepackage[all]{xy}
\usepackage{url}

\newtheorem{theorem}{Theorem}[section]
\newtheorem{lemma}[theorem]{Lemma}
\newtheorem{proposition}[theorem]{Proposition}

\theoremstyle{remark}
\newtheorem{remark}[theorem]{Remark}

\newtheorem{definition}[theorem]{Definition}
\newtheorem{notations}[theorem]{Notations}
\newtheorem*{conventions}{Conventions}

\newcommand{\T}{\mathscr{F}}

\newcommand{\q}{{\bf q}}

\newcommand{\p}{\partial}
\newcommand{\f}{{\bf f}}
\newcommand{\g}{{\bf g}}
\newcommand{\e}{\mathbf{e}}
\newcommand{\D}{\mathscr{D}}

\newcommand{\bart}{\bar{t}}

\begin{document}

\title[Invariance of Quantum Rings II]
{Invariance of Quantum Rings under Ordinary Flops II: \\ A quantum Leray--Hirsch theorem}

\author[Y.-P.~Lee]{Yuan-Pin~Lee}
\email{yplee@math.utah.edu}
\address{Y.-P.~Lee: Department of Mathematics, University of Utah,
Salt Lake City, Utah 84112-0090, U.S.A.}

\author[H.-W.~Lin]{Hui-Wen~Lin}
\email{linhw@math.ntu.edu.tw}
\address{H.-W.~Lin: Department of Mathematics and Taida
Institute of Mathematical Sciences (TIMS),
National Taiwan University, Taipei 10617, Taiwan}

\author[C.-L.~Wang]{Chin-Lung~Wang}
\email{dragon@math.ntu.edu.tw}
\address{C.-L. Wang: Department of Mathematics,
Center for Advanced Studies in Theoretical Sciences (CASTS), and Taida
Institute of Mathematical Sciences (TIMS),
National Taiwan University, Taipei 10617, Taiwan}


\begin{abstract}
This is the second of a sequence of papers proving the quantum invariance
for ordinary flops over an arbitrary smooth base.
In this paper, we complete the proof of the invariance of the big quantum
rings under \emph{ordinary flops of splitting type}.
 
To achieve that, several new ingredients are introduced.
One is a \emph{quantum Leray--Hirsch theorem}
for the local model (a certain toric bundle)
which extends the quantum $\D$ module of Dubrovin connection on the base by
a Picard--Fuchs system of the toric fibers.

Nonsplit flops as well as further applications of the quantum Leray--Hirsch
theorem will be discussed in subsequent papers. In particular, a \emph{quantum splitting principle} is developed in Part III \cite{LLQW} which reduces the general ordinary flops to the split case solved here.
\end{abstract}

\maketitle

\tableofcontents


\numberwithin{equation}{section}
\setcounter{section}{-1}

\section{Introduction}

\subsection{Overview}
This paper continues our study on quantum invariance of genus zero Gromov--Witten theory, up to analytic continuations along the K\"ahler moduli spaces, under ordinary flops over a non-trivial base. The quantum invariance via analytic continuations plays an important role in the study of various Calabi--Yau compactifications in string theory. It is also a potential tool in comparing various birational minimal models in higher dimensional algebraic geometry. We refer the readers to \cite{LLW} and Part I of this series \cite{LLWp1} for a general introduction.

In Part I, we had determined the \emph{defect of the cup product} 
under the canonical correspondence (I-\S1) and show that it is corrected 
by the small quantum product attached to the extremal ray (I-\S2). 
We then perform various reductions to the local models (I-\S3 and 4).
The most important consequence of this reduction is that we may
assume our ordinary flops are between two toric fibrations over the same
smooth base.

In this paper, we study the local models via various techniques and 
complete the proof of quantum invariance of Gromov--Witten theory
in genus zero under ordinary flops of splitting type.
This is, as far as we know, the first result on the quantum invariance under
the $K$-equivalence (crepant transformation) \cite{Wang2, Wang3} 
where the local structure of the exceptional loci can not be deformed to any 
explicit (e.g.~toric) geometry and the analytic continuation is nontrivial.
This is also the first result for which the analytic
continuation is established with nontrivial Birkhoff factorization.

Several new ingredients are introduced in the course of the proof.
One main technical ingredient is the \emph{quantum Leray--Hirsch theorem}
for the local model, which is related to the canonical lifting of the quantum
$\D$ module from the base to the total space of a (toric) bundle.
The techniques developed in this paper are applicable to more general cases and will be discussed in subsequent papers.

\begin{conventions}
This paper is strongly correlated with \cite{LLWp1}, which will
be referred to as ``Part I'' throughout the paper.
All conventions and notations there carry over to this paper (Part II).
\end{conventions}

\subsection{Outline of the contents}


\subsubsection{On the splitting assumption}

We recall the local geometry of an ordinary $P^r$ flop $f: X \dasharrow X'$ (Part~I \S 1.1). The local geometry of the $f$-exceptional loci $Z \subset X$ and $Z' \subset X'$ is encoded in a triple $(S, F, F')$, where $S$
is a smooth variety and $F$, $F'$ are two rank $r + 1$ vector
bundles over $S$. 
In Part I, we reduce the proof of the invariance of big quantum ring of
any ordinary flop to that of its local model.
Therefore, we may assume that
\[
 \begin{split}
 &X = \tilde{E} = P_{Z}(\mathscr{O}(-1) \otimes F' \oplus \mathscr{O}), \\
 &X' = \tilde{E}' = P_{Z'}(\mathscr{O}(-1) \otimes F \oplus \mathscr{O}),
 \end{split}
\]
where $Z \cong P_S(F)$ and $Z' \cong P_S(F')$ are projective bundles.
In particular, $X$ and $X'$ are toric bundles over the smooth base $S$. Moreover, it is equivalent to proving the \emph{type I quasi-linearity property}, namely the invariance for one pointed descendent fiber series of the form
$$
\langle \bart_1, \cdots, \bart_{n - 1}, \tau_k a \xi \rangle,
$$
where $\bar t_i \in H(S)$ and $\xi$ is the common infinity divisor of $X$ and $X'$.

To proceed, recall that the descendent GW invariants are encoded by their
generating function, i.e. the so called (big) $J$ function:
For $\tau \in H(X)$,
\begin{equation*}
  J^X(\tau, z^{-1}) := 1 + \frac{\tau}{z} + \sum_{\beta, n, \mu} \frac{q^{\beta}}{n!}
  T_{\mu} \left\langle \frac{T^{\mu}}{z(z-\psi)}, \tau, \cdots, \tau \right\rangle_{0,n+1, \beta}^X.
\end{equation*}
The determination of $J$ usually relies on the existence of $\Bbb
C^\times$ actions. Certain localization data $I_\beta$ coming from
the stable map moduli are of hypergeometric type. For ``good''
cases, say $c_1(X)$ is semipositive and $H(X)$ is generated by
$H^2$, $I(t) = \sum I_\beta\, q^\beta$ determines $J(\tau)$ on the
small parameter space $H^0 \oplus H^2$ through the ``classical''
\emph{mirror transform} $\tau = \tau(t)$. For a simple flop, $X =
X_{loc}$ is indeed semi-Fano toric and the classical Mirror Theorem
(of Lian--Liu--Yau and Givental) is sufficient \cite{LLW}. (It
turns out that $\tau = t$ and $I = J$ on $H^0 \oplus H^2$.)

For general base $S$ with given $QH(S)$, the determination of $QH(P)$ for
a projective bundle $P \to S$ is far more involved. To allow \emph{fiberwise localization} to determine the structure of GW
invariants of $X_{loc}$, the bundles $F$ and $F'$ are then assumed to be split bundles. 

In this paper (Part II), we only consider ordinary flops of 
\emph{splitting type}, 
namely $F \cong \bigoplus_{i = 0}^r L_i$ and $F' \cong \bigoplus_{i = 0}^r L_i'$
for some line bundles $L_i$ and $L_i'$ on $S$.

\subsubsection{Birkhoff factorization and generalized mirror transformation}

The splitting assumption allows one to apply the $\mathbb{C}^\times$
localizations along the fibers of the toric bundle $X \to S$.
Using this and other sophisticated technical tools,
J.~Brown (and A.~Givental) \cite{jB} proved that the
\emph{hypergeometric modification}
\begin{equation*}
I^X(D, \bar t, z, z^{-1}) := \sum_{\beta} q^\beta e^{\frac{D}{z} + (D.\beta)}I^{X/S}_\beta(z, z^{-1}) \bar\psi^*
J^S_{\beta_S}(\bar t, z^{-1})
\end{equation*}
lies in Givental's \emph{Lagrangian cone} generated by $J^X(\tau, z^{-1})$.
Here $D = t^1 h + t^2 \xi$, $\bar t \in H(S)$, $\beta_S = \bar\psi_* \beta$, and
the explicit form of $I^{X/S}_{\beta}$ is given in \S\ref{I-function}.

Based on Brown's theorem, we prove the following theorem.
(See \S\ref{IJ-BF/GMT} for notations on higher derivatives $\p^{z\e}$'s.)

\begin{theorem}[(BF/GMT)] \label{thm:0.4}
There is a unique matrix factorization
$$
 (\p^{z\e} I(z, z^{-1})) = (z\nabla J(z^{-1})) B(z),
$$
called the \emph{Birkhoff factorization} (BF) of $I$,
valid along $\tau  = \tau(D, \bar t, q)$.
\end{theorem}

BF can be stated in another way.
There is a recursively defined polynomial differential operator
$P(z, q; \p) = 1 + O(z)$ in $t^1, t^2$ and $\bar t$ such that
$$J (z^{-1})= P(z, q; \p) I(z, z^{-1}).$$
In other words, $P$ removes the $z$-polynomial part of $I$ in the
$NE(X)$-adic topology.
In this form, the \emph{generalized mirror transform} (GMT)
\begin{equation*}
\tau(D, \bar t, q) = D + \bar t + \sum_{\beta \ne 0} q^\beta \tau_\beta(D, \bar t)
\end{equation*}
is the coefficient of $z^{-1}$ in $J = PI$.

\subsubsection{Hypergeometric modification and $\D$ modules} 

In principle, knowing BF, GMT and GW invariants on $S$ allows us to calculate all $g = 0$ invariants on $X$ and $X'$
by reconstruction.
These data are in turn encoded in the $I$-functions.
One might be tempted to prove the $\T$-invariance by comparing
$I^X$ and $I^{X'}$.
While they are rather symmetric-looking,
the defect of cup product implies $\T I^X \neq I^{X'}$ and
the comparison via tracking the defects of ring isomorphism becomes hopelessly
complicated. This can be overcome by studying a more ``intrinsic'' object:
the cyclic $\D$ module $\mathscr{M}_J = \D J$, where $\D$ denotes the ring of
differential operators on $H$ with suitable coefficients.

It is well known by the topological recursion relations (TRR) that $(z\p_\mu J)$
forms a \emph{fundamental solution matrix} of the Dubrovin connection: 
Namely we have the \emph{quantum differential equations} (QDE)
$$
z\p_\mu z\p_\nu J = \sum _\kappa \tilde C_{\mu\nu}^\kappa (t)\, z\p_\kappa J,
$$
where $\tilde C_{\mu\nu}^\kappa (t) = \sum_\iota g^{\kappa\iota} \p^3_{\mu\nu\iota} F_0(t)$ are the structural constants of $*_t$. This implies that $\mathscr{M}$ is a \emph{holonomic} $\D$ module of length $N = \dim H$. For $I$ we consider a similar $\D$ module $\mathscr{M}_I = \D I$. The BF/GMT theorem furnishes a change of basis which implies that $\mathscr{M}_I$ is also holonomic of length $N$.

The idea is to go backward: To find $\mathscr{M}_I$ first and then transform it
to $\mathscr{M}_J$. 
We do not have similar QDE since $I$ does not have enough variables. 
Instead we construct higher order \emph{Picard--Fuchs equations} 
$\Box_\ell I = 0$, $\Box_\gamma I = 0$ in divisor variables, with the nice 
property that ``up to analytic continuations'' they generate $\T$-invariant 
ideals:
$$
\T \langle \Box^X_\ell, \Box^X_\gamma\rangle \cong \langle \Box^{X'}_{\ell'}, 
\Box^{X'}_{\gamma'}\rangle.
$$

\subsubsection{Quantum Leray--Hirsch and the conclusion of the proof} 
Now we want to determine $\mathscr{M}_I$.
While the derivatives along the fiber directions are determined by the
Picard--Fuchs equations, we need to find the derivatives along the base 
direction.
Write $\bar t = \sum \bar t^i \bar T_i$. 
This is achieved by \emph{lifting} the QDE on $QH(S)$, namely
$$
z\p_{i} z\p_{j} J^S = \sum _k \bar C_{ij}^k(\bar t)\, z\p_{k} J^S,
$$
to a differential system on $H(X)$. A key concept needed for such a lifting 
is the $I$-minimal lift
of a curve class $\beta_S \in NE(S)$ to $\beta_S^I \in NE(X)$.
Various lifts of curve classes are discussed in Section~\ref{s:6}.
See in particular Definition~\ref{d:imlift}.

Using Picard--Fuchs and the lifted QDE, we show that $\T \mathscr{M}_{I^X} \cong \mathscr{M}_{I^{X'}}$.
\begin{theorem} [(Quantum Leray--Hirsch)] \label{thm:0.5} {\ }
	
\begin{itemize}
\item[(1)]
($I$-Lifting) The quantum differential equation on $QH(S)$ can be lifted to $H(X)$ as
\begin{equation*}
z\p_i \,z\p_j I = \sum_{k, \beta_S} q^{\beta_S^I} e^{(D.\beta_S^I)}\bar C_{ij, \beta_S}^k(\bar t)\, z\p_k D_{\beta^I_S}(z) I,
\end{equation*}
where $D_{\beta^I_S}(z)$ is an operator depending only on $\beta_S^I$.
Any other lifting is related to it modulo the Picard--Fuchs system.
\item[(2)]
Together with the Picard--Fuchs $\Box_\ell$ and $\Box_\gamma$,
they determine a first order matrix system under the naive quantization
$\p^{z\e}$ (Definition~\ref{d:7.7}) of
canonical basis (Notations~\ref{n:1}) $T_\e$'s of $H(X)$:
$$
z\p_a (\p^{z\e} I) = (\p^{z\e} I) C_a(z, q), \qquad \mbox{where $t^a= t^1$, $t^2$ or $\bar t^i$}.
$$
\item[(3)]
The system has the property that for any fixed $\beta_S \in NE(S)$, the coefficients are formal functions in $\bar t$ and polynomial functions in $q^\gamma e^{t^2}$, $q^\ell e^{t^1}$ and $\f(q^\ell e^{t^1})$. Here the basic rational function
\begin{equation} \label{rat-f}
\f(q) := q/(1 - (-1)^{r  + 1}q)
\end{equation}
is the ``origin of analytic continuation'' satisfying $\f(q) + \f(q^{-1}) = (-1)^r$.
\item[(4)]
The system is $\T$-invariant.
\end{itemize}
\end{theorem}

The final step is to go from $\mathscr{M}_I$ to $\mathscr{M}_J$.
From the perspective of $\D$ modules, the BF can be considered as
a \emph{gauge transformation}.
The defining property $(\p^{z\e} I) = (z\nabla J) B$ of $B$ can be rephrased as
$$z\p_a (z\nabla J) = (z\nabla J)\tilde C_a$$
such that
\begin{equation} \label{e:0.1}
  \tilde C_a = (-z \p_a B + BC_a)B^{-1}
\end{equation}
is independent of $z$.

This formulation has the advantage that all objects in \eqref{e:0.1} are expected to be $\T$-invariant (while $I$ and $J$ are not).
It is therefore easier to first establish the $\T$-invariance of $C_a$'s and
use it to derive the $\T$-invariance of BF and GMT. As a consequence, this allows to deduce the type I quasi-linearity (Proposition \ref{red-GMT}), and hence the invariance of big quantum rings for local models.

\begin{theorem}[(Quantum invariance)] \label{main-thm}
For ordinary flops of splitting type, the big quantum cohomology ring is invariant up to analytic continuations.
\end{theorem}

By the reduction procedure in Part I, this is equivalent to the quasi-linearity property
of the local models. This completes the outline.

\begin{remark}
Results in this paper had been announced, in increasing degree of generalities,
by the authors in various conferences during 2008-2012; 
see e.g.~\cite{Lin, Wang4, LLW2, LW} where more example-studies can be found. Examples on quantum Leray--Hirsch are included in \S \ref{MSJ-ex}. The complete proofs of Theorem \ref{thm:0.5} and \ref{main-thm} were achieved in mid-2011. 

It might seem possible to prove Theorem \ref{main-thm} directly 
from comparisons of $J$-functions and Birkhoff factorizations on $X$ and $X'$.
Indeed, we were able to carry this out for various special cases. Mysterious \emph{regularization phenomenon} appears during such a direct approach. In the Appendix 
we explain how regularization of certain rational functions leads to the beginning steps of analytic continuations in our context. However the combinatorial complexity becomes intractable (to us) in the general case. Some examples can be found in the proceedings articles referred above.
\end{remark}

In Part III \cite{LLQW}, the final part of this series, we will develop a quantum splitting principle 
to remove the \emph{splitting assumption} in Theorem \ref{main-thm}. This then completes our study on the quantum invariance under ordinary flops.

\subsection{Acknowledgements}
Y.-P.\ Lee is partially supported by the NSF; H.-W.\ Lin is partially supported by the MOST; C.-L.\ Wang is partially supported by the MOST and the MOE.
We are particularly grateful to Taida Institute of Mathematical Sciences (TIMS)
for its steady support which makes this long-term collaborative project
possible. We would also like to thank the anonymous referee for pointing out several typographical errors in an earlier version of the paper.

\section{Birkhoff factorization} \label{IJ-BF/GMT}

In this section, a general framework for calculating the $J$ function for a 
split toric bundle is discussed.
It relies on a given (partial) section $I$ of the Lagrangian cone generated by 
$J$. 
The process to go from $I$ to $J$ is introduced in a constructive manner, 
and Theorem \ref{thm:0.4} will be proved 
($=$ Proposition \ref{BF0} + Theorem \ref{BF/GMT}).

\subsection{Lagrangian cone and the $J$ function}

We start with Givental's symplectic space reformulation of Gromov--Witten 
theory arising from the
\emph{dilaton, string, and topological recursion relation}. 
The main references for this section are \cite{aG1, CG},
with supplements and clarification from \cite{LP2, ypL}.
In the following, the underlying ground ring is the
\emph{Novikov ring}
$$
R =\widehat{\mathbb{C}[NE(X)]}.
$$
All the complicated issues on completion are deferred to \cite{LP2}.

Let $H:= H(X)$, $\mathcal{H} := H[z,z^{-1}]\!]$, $\mathcal{H}_+ :=H[z]$
and $\mathcal{H}_- := z^{-1}H[\![z^{-1}]\!]$.
Let $1 \in H$ be the identity.
One can identify $\mathcal{H}$ as $T^*\mathcal{H}_+$ and this gives a
canonical symplectic structure and a vector bundle structure on $\mathcal{H}$.

Let
\[
  \q(z) = \sum_{\mu} \sum_{k=0}^{\infty} \q^{\mu}_k T_{\mu} z^k  \in \mathcal{H}_+
\]
be a general point, where $\{ T_{\mu} \}$ form a basis of $H$.
In the Gromov--Witten context, the natural coordinates on $\mathcal{H}_+$
are ${\bf t}(z) = \q(z) + 1 z$ (dilaton shift),
with ${\bf t}(\psi) = \sum_{\mu, k} t^{\mu}_k T_{\mu} \psi^k$
serving as the general descendent insertion.
Let $F_0({\bf t})$ be the generating function of genus zero \emph{descendent} Gromov--Witten invariants on $X$.
Since $F_0$ is a function on $\mathcal{H}_+$, the one form $dF_0$ gives a section of
$\pi:\mathcal{H} \to \mathcal{H}_+$.

Givental's \emph{Lagrangian cone} $\mathcal{L}$ is defined as the graph of $dF_0$,
which is considered as a section of $\pi$.
By construction it is a Lagrangian subspace.
The existence of $\mathbb{C}^*$ action on $\mathcal{L}$
is due to the dilaton equation $\sum \q^\mu_k {\p}/{\p \q^\mu_k} F_0 = 2 F_0$. Thus $\mathcal{L}$ is a cone with vertex $\q = 0$ (c.f.~\cite{aG1, ypL}).

Let $\tau = \sum_{\mu} \tau^{\mu} T_{\mu} \in H$.
Define the (big) $J$-function to be
\begin{equation} \label{e:5.1}
 \begin{split}
  J^X(\tau, z^{-1}) &= 1 + \frac{\tau}{z} + \sum_{\beta, n, \mu} \frac{q^{\beta}}{n!}
  T_{\mu} \left\langle \frac{T^{\mu}}{z(z-\psi)}, \tau, \cdots, \tau \right\rangle_{0,n+1, \beta}\\
 &= e^{\frac{\tau}{z}} + \sum_{\beta \neq 0, n, \mu} \frac{q^{\beta}}{n!}
  e^{\frac{\tau_1}{z} + (\tau_1.\beta)} T_{\mu} \left\langle \frac{T^{\mu}}{z(z-\psi)}, \tau_2, \cdots, \tau_2 \right\rangle_{0,n+1, \beta} ,
 \end{split}
\end{equation}
where in the second expression $\tau  = \tau_1 + \tau_2$ with $\tau_1 \in H^2$. The equality follows from the divisor equation for descendent invariants. Furthermore, the string equation for $J^X$ says that we can take out the fundamental class $1$ from the variable $\tau$ to get an overall factor $e^{\tau^0/z}$ in front of (\ref{e:5.1}).

The $J$ function can be considered as a map from $H$ to $z\mathcal{H}_-$.
Let $L_{\mathbf{f}} = T_{\mathbf{f}}\mathcal{L}$ be the tangent space of $\mathcal{L}$ at $\mathbf{f} \in \mathcal{L}$.
Let $\tau \in H$ be embedded into $\mathcal{H}_+$ via
\[
  H \cong -1 z + H \subset \mathcal{H}_+ .
\]
Denote by  $L_{\tau} = L_{(\tau, dF_0(\tau))}$. Here we list the basic structural results from \cite{aG1}:

\begin{itemize}
\item[(i)] $zL_{\tau} \subset L_{\tau}$ and so $L_{\tau}/ zL_{\tau} \cong \mathcal{H}_+/z\mathcal{H}_+ \cong H$ has rank $N := \dim H$.

\item[(ii)] $L_{\tau} \cap \mathcal{L} = zL_{\tau}$, considered as subspaces inside $\mathcal{H}$.

\item[(iii)] The subspace $L_{\tau}$ of $\mathcal{H}$ is the tangent space at
every $\mathbf{f} \in zL_{\tau} \subset \mathcal{L}$. Moreover, $T_{\mathbf{f}} = L_{\tau}$ implies that $\mathbf{f} \in zL_{\tau}$.
$zL_{\tau}$ is considered as the \emph{ruling} of the cone.

\item[(iv)] The intersection of $\mathcal{L}$ and the affine space
$- 1z + z\mathcal{H}_-$
is parameterized by its image $ -1 z + H \cong H \ni \tau$
via the projection by $\pi$.
$$
-z J^X(\tau, -z^{-1}) = -1z + \tau + O(1/z)
$$
is the function of $\tau$ whose graph is the intersection.

\item[(v)] The set of all directional derivatives $z\p_\mu J^X = T_\mu + O(1/z)$ spans an $N$ dimensional subspace of $L_{\tau}$, namely $L_{\tau} \cap z \mathcal{H}_-$, such that
its projection to $L_{\tau}/zL_{\tau}$ is an isomorphism.
\end{itemize}

Note that we have used the convention of the $J$ function which differs from
that of some more recent papers \cite{aG1, CG} by a factor $z$.

\begin{lemma} \label{l:1}
$z \nabla J^X = (z \partial_{\mu} J^{\nu})$ forms a matrix whose column
vectors $z \partial_{\mu} J^X (\tau)$ generates the tangent space $L_{\tau}$
of the Lagrangian cone $\mathcal{L}$ as an $R \{z\}$-module. Here $a = \sum q^\beta a_\beta(z)\in R\{z\}$ if $a_\beta(z) \in \Bbb C[z]$.
\end{lemma}

\begin{proof}
Apply (v) to $L_{\tau}/zL_{\tau}$ and multiply $z^k$ to get $z^kL_{\tau}/z^{k+1} L_{\tau}$.
\end{proof}

We see that the germ of $\mathcal{L}$ is determined by an
$N$-dimensional submanifold.
In this sense, $zJ^X$ \emph{generates} $\mathcal{L}$.
Indeed, all discussions are applicable to the Gromov--Witten context only as
\emph{formal germs} around the neighborhood of $\q =-1z$.

\subsection{Generalized mirror transform for toric bundles} \label{s:5.2}
Let $\bar p : X \to S$ be a smooth fiber bundle such that $H(X)$ is generated
by $H(S)$ and \emph{fiber divisors} $D_i$'s as an algebra,
such that there is no linear relation among $D_i$'s and $H^2(S)$.
An example of $X$ is a toric bundle over $S$.
Assume that $H(X)$ is a free module over $H(S)$ with finite generators
$\{ D^e := \prod_{i} D_i^{e_i} \}_{e \in \Lambda^+}$.

Let $\bar t := \sum_s \bar t^s \bar T_s$ be a general cohomology class in $H(S)$,
which is identified with $\bar p^* H(S)$.
Similarly denote  $D = \sum t^i D_i$ the general fiber divisor.
Elements in $H(X)$ can be written as linear combinations of $\{ T_{(s, e)} = \bar T_s D^e \}$. Denote the $\bar T_s$ directional derivative on $H(S)$ by $\p_{\bar T_s} \equiv \p_{\bar t^s}$, and denote the
multiple derivative
\[
 \partial^{(s, e)} := \p_{\bar t^s} \prod_i
 \partial_{t^i}^{e_i} .
\]
Note, however, most of the time $z$ will appear with derivative. For the notational convenience, denote the index $(s, e)$ by $\e$. We then denote
\begin{equation} \label{zp}
\p^{z\e} \equiv \partial^{z(s, e)} := z\p_{\bar t^s} \prod_i
 z\partial_{t^i}^{e_i} = z^{|e| + 1}\partial^{(s, e)}.
\end{equation}

As usual, the $T_\e$ directional derivative on $H(X)$ is denoted by $\p_{\e} = \p_{T_\e}$. This is a special choice of basis $T_\mu$ (and $\p_\mu$) of $H(X)$, which is denoted by
$$
T_{\e} \equiv T_{(s, e)} \equiv \bar T_s D^e; \qquad \e \in \Lambda^+.
$$
The two operators $\p^{z\e}$ and $z\p_{\e}$ are by definition very different, nevertheless they are closely related in the study of quantum cohomology as we will see below.

Assuming that $\bar p : X \to S$ is a toric bundle of the split type, i.e.~toric
quotient of a split vector bundle over $S$. Let $J^S(\bar t, z^{-1})$ be the $J$ function on $S$. The hypergeometric modification of $J^S$ by the $\bar p$-fibration takes the form
\begin{equation} \label{HGM}
I^X(\bar t, D, z, z^{-1}) := \sum_{\beta \in NE(X)} q^\beta e^{\frac{D}{z} + (D.\beta)} I^{X/S}_\beta(z, z^{-1})
J^S_{\beta_S}(\bar t, z^{-1})
\end{equation}
with the relative factor $I^{X/S}_\beta$, whose explicit form for $X = \tilde{E} \to S$ will be given in Section~\ref{I-function}.

The major difficulty which makes $I^X$ being deviated from $J^X$ lies in the fact that in general positive $z$ powers may occur in $I^X$. Nevertheless for each $\beta \in NE(X)$, the power of $z$ in $I^{X/S}_\beta(z, z^{-1})$ is bounded above by a constant depending only on $\beta$. Thus we may study $I^X$ in the space $\mathcal{H} := H[z,z^{-1}]\!]$ over $R$.

Notice that the $I$ function is defined only in the
subspace
\begin{equation}
\hat t := \bar t + D \in H(S) \oplus \bigoplus_i \mathbb{C} D_i \subset H(X).
\end{equation}

We will use the following theorem by J.~Brown (and A.~Givental):

\begin{theorem}[(\cite{jB} Theorem~1)] \label{t:Brown}
$(-z)I^X(\hat t, -z)$ lies in the Lagrangian cone $\mathcal{L}$ of $X$.
\end{theorem}

\begin{definition} [GMT]
For each $\hat t$, $zI(\hat t)$ lies in $L_\tau$ of $\mathcal{L}$.
The correspondence
$$
\hat t \mapsto \tau(\hat t) \in H(X) \otimes R
$$
is called the \emph{generalized mirror transformation} (c.f.~\cite{CG, aG1}).
\end{definition}

\begin{remark}
In general $\tau(\hat t)$ may be outside the submodule of the Novikov ring $R$
generated by $H(S) \oplus \bigoplus_i \Bbb C D_i$.
This is in contrast to the (classical) mirror transformation where
$\tau$ is a transformation within $(H^0(X) \oplus H^2(X))_R$
(small parameter space).
\end{remark}

To make use of Theorem \ref{t:Brown}, we start by outlining the idea behind the following discussions. By the properties of $\mathcal{L}$, Theorem \ref{t:Brown} implies that $I$ can be obtained from $J$ by applying certain differential operator in $z\p_{\e}$'s to it, with coefficients being series in $z$. However, what we need is the reverse direction, namely to obtain $J$ from $I$, which amounts to removing the positive $z$ powers from $I$.  Note that, the $I$ function has variables only in the subspace $H(S) \oplus \bigoplus_i \mathbb{C} D_i$. Thus a priori the reverse direction does not seem to be possible.

The key idea below is to replace derivatives in the missing directions by higher order differentiations in the fiber divisor variables $t^i$'s, a process similar to transforming a first order ODE system to higher order scaler equation. This is possible since $H(X)$ is generated by $D_i$'s as an algebra over $H(S)$.

\begin{lemma} \label{ST}
$z \p_{1} J^X =J^X$ and $z \p_{1} I^X =I^X$.
\end{lemma}

\begin{proof}
The first one is the string equation.
For the second one, by definition 
$$
I^X = \sum_{\beta} q^\beta e^{D/z +(D.\beta)}  I^{X/S}_{\beta}J^S_{\beta_S}(\bar t),
$$ 
where
$I^{X/S}_{\beta}$ depends only on $z$. The differentiation with respect to $t^0$ (dual coordinate of $1$) only applies to $J^S_{\beta_S}(\bar t)$.
Hence the string equation on $J^S_{\beta_S}(\bar t)$ concludes the proof.
\end{proof}

To avoid cluttered notations, we use $I$ and $J$ to denote $I$-function and $J$-function respectively when the target space is clear.

\begin{proposition} \label{BF0}
{\rm (1)} The GMT: $\tau =\tau(\hat t)$ satisfies $\tau(\hat t, q = 0) = \hat t$.

{\rm (2)} Under the basis $\{T_{\e}\}_{\e \in \Lambda^+}$, there exists an \emph{invertible} $N \times N$ matrix-valued formal series $B(\tau, z)$, which is free from cohomology classes, such that
\begin{equation} \label{e:5.2}
 \left( \partial^{z\e} I(\hat t, z, z^{-1}) \right)
 =  \left( z \nabla J(\tau, z^{-1}) \right) B(\tau, z),
\end{equation}
where ${\displaystyle \left( \partial^{z\e} I \right) }$
is the $N \times N$ matrix with $\partial^{z\e} I$ as column vectors.
\end{proposition}

\begin{proof}
By Theorem~\ref{t:Brown},
$zI \in \mathcal{L}$, hence $z \p I \in T\mathcal{L} = L$. Then $z(z\p) I \in zL \subset \mathcal{L}$ and so $z\p (z\p) I$ lies again in $L$. Inductively, $\p^{z\e} I$ lies in $L$.
The factorization $(\p^{z\e}I) = (z\nabla J) B(z)$ then follows from Lemma~\ref{l:1}. Also Lemma \ref{ST} says that the $I$ (resp.~$J$) function appears as the first column vector of $(\p^{z\e}I)$ (resp.~$(z\nabla J)$). By the $R\{z\}$ module structure it is clear that $B$ does not involve any cohomology classes.

By the definitions of $J$, $I$ and $\p^{z\e}$ (c.f.~(\ref{e:5.1}),
(\ref{HGM}), (\ref{zp})), it is clear that
\begin{equation} \label{modq}
\p^{z\e} e^{\hat t/{z}} = T_{\e} e^{\hat t/z},\qquad z\p_{\e}
e^{{t}/{z}} = T_{\e} e^{t/z}
\end{equation}
($t\in H(X)$).
Hence modulo Novikov variables $\p^{z\e} I(\hat t) \equiv T_{\e} e^{\hat t/z}$ and $z{\partial}_{\e} J(\tau) \equiv T_{\e} e^{\tau/z}$

To prove (1), modulo all $q^\beta$'s we have
$$
e^{\hat t/z} \equiv \sum_{\e \in \Lambda^+} B_{\e, 1}(z) T_{\e} e^{\tau(\hat t)/z}.
$$
Thus
$$e^{(\hat t - \tau(\hat t))/z} \equiv \sum_{\e} B_{\e, 1}(z) T_{\e},$$
which forces that $\tau(\hat t) \equiv \hat t$
(and $B_{\e, 1}(z) \equiv \delta_{T_\e, 1}$).

To prove (2), notice that by (1) and (\ref{modq}), $B(\tau, z) \equiv I_{N \times N}$ when modulo Novikov variables, so in particular $B$ is invertible.  Notice that in getting (\ref{e:5.2}) we do not need to worry about the sign on ``$-z$'' since it appears in both $I$ and $J$.
\end{proof}

\begin{definition}[BF]
The left-hand side of \eqref{e:5.2} involves $z$ and $z^{-1}$,
while the right-hand side is the product of a function of $z$ and a function of
$z^{-1}$.
Such a matrix factorization process is termed the \emph{Birkhoff factorization}.
\end{definition}

Besides its existence and uniqueness, for actual computations it will be important to know how to calculate $\tau(\hat t)$ directly or inductively.

\begin{proposition} \label{p:5.5}
There are scalar-valued formal series $C_{\e}(\hat t, z)$
such that
\begin{equation} \label{e:5.3}
  J({\tau}, z^{-1}) =   \sum_{\e \in \Lambda^+} C_{\e}(\hat t, z) \,
  {\partial}^{z \e} I(\hat t, z, z^{-1}),
\end{equation}
where $C_{\e} \equiv \delta_{T_\e, 1}$ modulo Novikov variables.

In particular $\tau(\hat t) = \hat t + \cdots$ is determined by the $1/z$ coefficients of the RHS.
\end{proposition}

\begin{proof}
Proposition \ref{BF0} implies that
\[
 z \nabla J = \left( \partial^{z\e} I \right) B^{-1}.
\]
Take the first column vector in the LHS, which is $z \nabla_1 J = J$ by Lemma \ref{ST},
one gets expression (\ref{e:5.3}) by defining $C_\e$ to be the corresponding $\e$-th entry of the first column vector of $B^{-1}$. Modulo $q^\beta$'s, $B^{-1} \equiv I_{N \times N}$, hence  $C_\e \equiv \delta_{T_\e, 1}$.
\end{proof}

\begin{definition}
A differential operator $P$ is of degree $\Lambda^+$ if $P = \sum_{\e \in \Lambda^+} C_\e \p^{z\e}$ for some $C_\e$. Namely, its components are multi-derivatives indexed by $\Lambda^+$.
\end{definition}

\begin{theorem} [(BF/GMT)] \label{BF/GMT}
There is a unique, recursively determined, scalar-valued degree $\Lambda^+$ differential operator
$$
P(z) = 1 + \sum_{\beta \in NE(X) \backslash \{0\}} q^\beta P_\beta(t^i, \bar t^s, z; z\p_{t^i}, z \p_{\bar t^s}),
$$
with each $P_\beta$ being polynomial in $z$, such that $P(z) I(\hat t, z, z^{-1}) = 1 + O(1/z)$.

Moreover,
$$
J(\tau(\hat t), z^{-1})  = P(z) I(\hat t, z, z^{-1}),
$$
with $\tau(\hat t)$ being determined by the $1/z$ coefficient of the right-hand side.
\end{theorem}

\begin{proof}
The operator $P(z)$ is constructed by induction on
$\beta \in NE(X)$. We set $P_\beta = 1$ for $\beta = 0$. Suppose that $P_{\beta'}$ has been constructed for all $\beta' < \beta$ in $NE(X)$. We set $P_{< \beta}(z) = \sum_{\beta' < \beta} q^{\beta'} P_{\beta'}$. Let
\begin{equation} \label{top-z}
A_1 = z^{k_1} q^\beta \sum_{\e  \in \Lambda^+} f^{\e}(t^i, \bar t^s) T_{\e}
\end{equation}
be the top $z$-power term in $P_{< \beta}(z) I$. If $k_1 < 0$ then we are done. Otherwise  we will remove it by introducing ``certain $P_\beta$''.
Consider the ``naive quantization''
\begin{equation} \label{naive-q}
\hat A_1  := z^{k_1} q^\beta \sum_{\e  \in \Lambda^+} f^{\e}(t^i, \bar t^s) \p^{z\e}.
\end{equation}
In the expression
$$
(P_{< \beta}(z) - \hat A_1) I = P_{< \beta}(z) I - \hat A_1 I,
$$
the target term $A_1$ is removed since
$$
\hat A_1 I(q = 0) = \hat A_1 e^{\hat t/z} = A_1 e^{\hat t/z} = A_1 + A_1 O(1/z).
$$
All the newly created terms either have smaller $z$-power or have curve degree $q^{\beta''}$ with $\beta'' > \beta$ in $NE(X)$. Thus we may keep on removing the new top $z$-power term $A_2$, which has $k_2 < k_1$. Since the process will stop in no more than $k_1$ steps, we simply define $P_\beta$ by
$$
q^\beta P_\beta = -\sum_{1 \le j \le k_1} \hat A_j.
$$
By induction we get $P(z) = \sum_{\beta \in NE(X)} q^\beta P_\beta$, which is clearly of degree $\Lambda^+$.

Now we prove the uniqueness of $P(z)$. Suppose that $P_1(z)$ and $P_2(z)$ are two such operators. The difference $\delta(z) = P_1(z) - P_2(z)$ satisfies
$$
\delta(z) I =: \sum_{\beta} q^\beta \delta_\beta I = O(1/z).
$$
Clearly $\delta_0 = 0$. If $\delta_\beta \ne 0$ for some $\beta$, then $\beta$ can be chosen so that $\delta_{\beta'} = 0$ for all $\beta' < \beta$. Let the highest non-zero $z$-power term of $\delta_\beta$ be $z^k \sum_\e \delta_{\beta, k, \e} \p^{z\e}$. Then
$$
q^\beta z^k \sum_\e \delta_{\beta, k, \e} \p^{z\e} \Big(e^{\hat t/z} + \sum_{\beta_1 \ne 0} q^{\beta_1} I_{\beta_1}\Big) + R I = O(1/z).
$$
Here $R$ denotes the remaining terms in $\delta$. Note that terms in $RI$ either do not contribute to $q^\beta$ or have $z$-power smaller than $k$. Thus the only $q^\beta$ term is
$$
q^\beta z^k \sum_\e \delta_{\beta, k, \e} T_\e e^{\hat t/z}.
$$
This is impossible since $k \ge 0$ and $\{T_\e\}$ is a basis. Thus $\delta = 0$.

Finally, by Lemma \ref{l:1} $B$, and so does $B^{-1}$, has entries in $R\{z\}$. Thus Proposition \ref{p:5.5} provides an operator which satisfies the required properties. By the uniqueness it must coincide with the effectively constructed $P(z)$.
\end{proof}

\subsection{Reduction to special BF/GMT}

\begin{proposition} \label{red-GMT}
Let $f: X \dasharrow X'$ be the projective local model of an ordinary flop with graph correspondence $\T$. Suppose there are formal liftings $\tau$, $\tau'$ of $\hat t$ in $H(X) \otimes R$ and $H(X') \otimes R$ respectively, with $\tau(\hat t), \tau'(\hat t) \equiv \hat t$ when modulo Novikov variables in $NE(S)$, and with $\T \tau(
\hat t) \cong \tau'(\hat t)$. Then
$$
\T J(\tau(\hat t)).\xi \cong J'(\tau'(\hat t)).\xi' \Longrightarrow \T J(\hat t).\xi \cong J'(\hat t).\xi'
$$
and consequently $QH(X)$ and $QH(X')$ are analytic continuations to each other under $\T$.
\end{proposition}

\begin{proof}
By induction on the weight $w:=(\beta_S,d_2) \in W$, suppose that for all $w' < w$ we have invariance of any $n$-point function (except that if $\beta'_S = 0$ then $n \ge 3$). Here we would like to recall  that $W := (NE(\tilde E)/\sim) \subset NE(S) \oplus \mathbb{Z}$ is the quotient Mori cone.

By the definition of $J$ in (\ref{e:5.1}), for any $a \in H(X)$ we may pick up the fiber series over $w$ from the $\xi a z^{-(k + 2)}$ component of the assumed $\T$-invariance:
\begin{equation} \label{xi-equiv}
\T \langle \tau^n, \psi^k \xi a\rangle^X \cong \langle \tau'^n, \psi^k \xi' \T a \rangle^{X'}.
\end{equation}

Write $\tau(\hat t) = \sum_{\bar w \in W} \tau_{\bar w}(\hat t) q^{\bar w}$. The fiber series is decomposed into sum of subseries in $q^\ell$ of the form
$$
\langle \tau_{\bar w_1}(\hat t), \cdots, \tau_{\bar w_n}(\hat t), \psi^k\xi a\rangle^X_{w''} q^{\sum_{j = 1}^n \bar w_j + w''}.
$$
Since $\sum \bar w_j + w'' = w$, any $\bar w_j \ne 0$ term leads to $w" < w$, whose fiber series is of the form $\sum_i g_i(q^\ell, \hat t) h_i(q^\ell)$ with $g_i$ from $\prod \tau_{\bar w_j}(\hat t)$ and $h_i$ a fiber series over $w"$. The $g_i$ is $\T$-invariant by assumption and $h_i$ is $\T$-invariant by induction, thus the sum of products is also $\T$-invariant.

From (\ref{xi-equiv}) and $\tau_0(\hat t) = \hat t$, the remaining fiber series with all $\bar w_j = 0$ satisfies
$$
\T \langle \hat t^n, \psi^k \xi a\rangle^X_{w} \cong \langle \hat t^n, \psi^k \xi \T a\rangle^{X'}_{w'},
$$
which holds for any $n$, $k$ and $a$.

Now by Part I Theorem~4.2 (divisorial reconstruction and WDVV reduction) 
this implies the $\T$-invariance of all fiber series over $w$.
\end{proof}

Later we will see that for the GMT $\tau(\hat t)$ and $\tau'(\hat t)$, the lifting condition $\tau(\hat t) \equiv \hat t$ modulo $NE(S) \backslash \{0\}$ (instead of modulo $NE(X) \backslash \{0\}$) and the \emph{identity} $\T J(\tau(\hat t)).\xi \cong J'(\tau'(\hat t)).\xi'$ holds for split ordinary flops.

\section{Hypergeometric modification} \label{s:6}

From now on we work with a split local $P^r$ flop $f: X \dasharrow X'$ with bundle data $(S, F, F')$, where
$$
F = \bigoplus_{i = 0}^r L_i \quad \text{and} \quad F' = \bigoplus_{i = 0}^r L_i'.
$$
We study the explicit formula of the hypergeometric modification $I^X$ and
$I^{X'}$ associated to the double projective bundles $X \to S$ and $X' \to S$,
especially the symmetry property between them.

In order to get a better sense of the factor $I^{X/S}$ it is necessary to have
a precise description of the Mori cone first.
We then describe the Picard--Fuchs equations associated to the $I$ function.

\subsection{The minimal lift of curve classes and $\T$-effective cone}
Let $C$ be an irreducible projective curve with $\psi: V = \bigoplus_{i = 0}^r \mathscr{O}(\mu_i) \to C$ a split bundle. Denote by $\mu = \max \mu_i$ and $\bar\psi: P(V) \to C$ the associated projective bundle. Let $h
= c_1(\mathscr{O}_{P(V)}(1))$,
\begin{equation*}
b = \bar \psi^*[C].H_r = H_r = h^r + c_1(V)h^{r - 1}
\end{equation*}
be the canonical lift of the base curve, and $\ell$ be the fiber
curve class.

\begin{lemma} \label{minimal lift}
$NE(P(V))$ is generated by $\ell$ and $b - \mu \ell$.
\end{lemma}

\begin{proof}
Consider $V' = \mathscr{O}(-\mu) \otimes V = \mathscr{O} \oplus
N$. Then $N$ is a semi-negative bundle and $NE(P(V)) \cong
NE(P(V'))$ is generated by $\ell$ and the zero section $b'$ of $N
\to P^1$. In this case $b'$ is also the canonical lift $b' = h'^r
+ c_1(V')h'^{r - 1}$. From the Euler sequence we know that $h' = h + \mu p$. Hence
\begin{equation*}
\begin{split}
b' &= (h + \mu p)^r + \sum_{i = 1}^r (\mu_i - \mu)p (h + \mu p)^{r + 1} \\
&= h^r + r\mu p h^{r - 1} + \sum_{i = 1}^r (\mu_i - \mu)p h^{r - 1}\\
&= h^r + c_1(V)h^{r - 1} - \mu ph^{r - 1}\\
&= b - \mu \ell.
\end{split}
\end{equation*}
\end{proof}

Let $\psi: V = \bigoplus_{i = 0}^r L_i \to S$ be a split bundle
with $\bar\psi: P = P(V) \to S$. Since $\bar\psi_*: NE(P) \to NE(S)$ is surjective, for each $\beta_S \in NE(S)$ represented by a curve $C = \sum_j n_j C_j$, the determination of $\bar\psi_*^{-1}(\beta_S)$ corresponds to the determination of $NE(P(V_{C_j}))$ for all $j$. Therefore by
Lemma \ref{minimal lift}, the minimal lift with respect to this curve decomposition is given by
\begin{equation*}
\beta^P := \sum_j n_j(\bar\psi^* [C_j].H_r - \mu_{C_j}\ell) = \beta_S - \mu_{\beta_S} \ell,
\end{equation*}
with $\mu_{C_j} = \max_i (C_j.L_i)$ and $\mu = \mu_{\beta_S} := \sum_j n_j \mu_{C_j}$. As before we identify the canonical lift $\bar\psi^*
\beta_S.H_r$ with $\beta_S$. Thus the crucial part is to determine the case of primitive classes. The general case follows from the primitive case by additivity. When there are more than one way to decompose into primitive classes, the minimal lift is obtained by taking the minimal one. Notice that further decomposition leads to smaller (or equal) lift. Also there could be more than one minimal lifts coming from different (non-comparable) primitive decompositions.

Now we apply the above results to study the effective and $\T$-effective
curve classes
under local split ordinary flop $f: X \dasharrow X'$ of type $(S, F, F')$.
Fixing a \emph{primitive} curve class $\beta_S \in NE(S)$, we define
\begin{equation*}
\mu_i := (\beta_S.L_i), \quad\mu'_i := (\beta_S.L'_i).
\end{equation*}
Let $\mu = \max \mu_i$ and
$\mu' = \max \mu'_i$. Then by working on an irreducible representation curve $C$ of $\beta_S$, we get by Lemma \ref{minimal lift}
\begin{equation*}
\begin{split}
NE(Z)_{\beta_S} &= (\beta_S - \mu \ell) + \Bbb Z_{\ge 0} \ell \equiv \beta_Z + \Bbb Z_{\ge 0} \ell,\\
NE(Z')_{\beta_S} &= (\beta_S - \mu' \ell') + \Bbb Z_{\ge 0} \ell' \equiv \beta_{Z'} + \Bbb Z_{\ge 0} \ell'.
\end{split}
\end{equation*}

Now we consider the further lift of the primitive element $\beta_Z$
(resp.~$\beta_{Z'}$) to $X$. The bundle $N \oplus
\mathscr{O}$ is of splitting type with Chern roots
$-h + L_i'$ and $0$, $i = 0, \ldots, r$. On $\beta_Z$ they take values
\begin{equation} \label{roots}
\mu + \mu'_i\quad (i = 0, \ldots, r) \quad\mbox{and}\quad 0.
\end{equation}
To determine the minimal lift of $\beta_Z$ in $X$, we separate it into two cases:

Case (1): $\mu + \mu' > 0$. The largest number in (\ref{roots}) is
$\mu + \mu'$ and
\begin{equation*}
NE(X)_{\beta_Z} = (\beta_Z - (\mu + \mu')\gamma) + \Bbb Z_{\ge 0}
\gamma.
\end{equation*}

Case (2): $\mu + \mu' \le 0$. The largest number in (\ref{roots}) is
$0$ and
\begin{equation*}
NE(X)_{\beta_Z} = \beta_Z + \Bbb Z_{\ge 0} \gamma.
\end{equation*}

To summarize, we have
\begin{lemma} \label{ef}
Given a primitive class $\beta_S \in NE(S)$,  $\beta = \beta_S +
d\ell + d_2 \gamma \in NE(X)$ if and only if
\begin{equation} \label{constraint}
d \ge - \mu \quad \mbox{and} \quad d_2 \ge - \nu,
\end{equation}
where $\nu = \max\{\mu + \mu', 0\}$.
\end{lemma}

\begin{remark} \label{mu-nu}
For the general case $\beta_S = \sum_j n_j [C_j]$, the constants $\mu$, $\nu$ are replaced by
$$
\mu = \mu_{\beta_S} := \sum_j n_j \mu_{C_j}, \qquad \nu = \nu_{\beta_S} := \sum_j n_j \max \{\mu_{C_j} + \mu_{C_j'}, 0 \}.
$$
Thus a \emph{geometric minimal lift} $\beta_S^X \in NE(X)$ for
$\beta_S \in NE(S)$, \emph{with respect to the given primitive decomposition},
is
\begin{equation*}
\beta_S^X = \beta_S -\mu \ell - \nu\gamma.
\end{equation*}
(If $\mu_{C_j} + \mu'_{C_j} \ge 0$ for all $j$, then $\nu = \mu + \mu'$.)

The geometric minimal lifts describe $NE(X)$.
We will however only need a ``generic lifting'' ($I$-minimal lift in
Definition~\ref{d:imlift})
in the study of GW invariants.
\end{remark}

\begin{definition}
A class $\beta \in N_1(X)$ is $\T$-effective if $\beta \in
NE(X)$ and $\T \beta \in NE(X')$.
\end{definition}

\begin{proposition} \label{F-ef}
Let $\beta_S \in NE(S)$ be primitive. A class $\beta \in NE(X)$ over $\beta_S$ is $\T$-effective if and only if
\begin{equation}
d + \mu \ge 0 \quad \mbox{and} \quad d_2 - d + \mu' \ge 0.
\end{equation}
\end{proposition}

\begin{proof}
Let $\beta = \beta_S + d\ell + d_2 \gamma$, then $\T \beta = \beta_S - d\ell' + d_2 (\gamma' + \ell') = \beta_S + (d_2 - d)\ell' + d_2 \gamma =: \beta_S + d'\ell' + d'_2\gamma'$. It is clear that $\beta$ is $\T$-effective implies both inequalities. Conversely, the two inequalities imply that
$$
d_2 \ge d - \mu' \ge -(\mu + \mu') \ge -\nu,
$$
hence $\beta \in NE(X)$. Similarly $\T \beta \in NE(X')$.
\end{proof}

\subsection{Symmetry for $I$} \label{I-function}

For $F = \bigoplus_{i = 0}^r L_i$, $F' = \bigoplus_{i = 0}^r L'_i$, the Chern polynomials for $F$ and $N \oplus \mathscr{O}$ take the form
$$
f_F = \prod a_i := \prod (h + L_i), \qquad f_{N \oplus \mathscr{O}} =
b_{r + 1}\prod b_i := \xi \prod (\xi - h + L'_i).
$$
For $\beta =
\beta_S + d \ell + d_2 \gamma$, we set $\mu_i :=
(L_i.\beta_S)$, $\mu_i' := (L_i'.\beta_S)$. Then for $i = 0, \ldots, r$, $
(a_i.\beta) = d + \mu_i$, $(b_i.\beta) = d_2 - d + \mu_i'$, and $(b_{r + 1}.\beta) = d_2$. Let
\begin{equation}
\lambda_\beta = (c_1(X/S).\beta) = (c_1(F) + c_1(F')).\beta_S + (r + 2)d_2.
\end{equation}
The relative $I$ factor is given by
\begin{equation} \label{I-gamma}
I^{X/S}_{\beta} := \frac{1}{z^{\lambda_\beta}} \frac{\Gamma(1 +
\frac{\xi}{z})}{\Gamma(1 + \frac{\xi}{z} + d_2)} \prod_{i = 0}^r
\frac{\Gamma(1 + \frac{a_i}{z})}{\Gamma(1 + \frac{a_i}{z} + \mu_i + d)}
\frac{\Gamma(1 + \frac{b_i}{z})}{\Gamma(1 + \frac{b_i}{z} + \mu_i' + d_2 - d)},
\end{equation}
and the hypergeometric modification of $\bar p: X \to S$ is
\begin{equation} \label{HG-mod}
I = I(D, \bar t; z, z^{-1}) = \sum_{\beta
\in NE(X)} q^\beta e^{\frac{D}{z} + (D.\beta)} I^{X/S}_{\beta}
J^S_{\beta_S}(\bar t),
\end{equation}
where $D = t^1h + t^2\xi$ is the fiber divisor and
$\bar t \in H(S)$.

In more explicit terms, for a split projective bundle 
$\bar\psi: P = P(V) \to S$, the
relative $I$ factor is
\begin{equation}
I^{P/S}_\beta := \prod_{i = 0}^r 
\frac{1}{\prod\limits_{m =1}^{\beta.(h + L_i)} (h + L_i + mz)},
\end{equation}
where the product in $m \in \mathbb{Z}$ is directed in the sense
that
\begin{equation}
\prod_{m = 1}^{s} 
 := \prod_{m = -\infty}^s/\prod_{m = -\infty}^0.
\end{equation}
Thus for each $i$ with $\beta.(h + L_i) \le -1$, the
corresponding subfactor is understood as in the numerator which must contain the factor $h + L_i$ corresponding to $m = 0$. In
general $I$ is viewed as a cohomology valued Laurent series in
$z^{-1}$. By the dimension constraint it in fact has only finite
terms.

\begin{remark}
The relative factor comes from the equivariant Euler class of $$H^0(C, T_{P/S}|_C) - H^1(C, T_{P/S}|_C)$$ at the moduli point $[C \cong P^1\to X]$.
\end{remark}

\begin{definition}[$I$-minimal lift] \label{d:imlift}
Introduce
\[
 \mu_{\beta_S}^I := \max_i \{\beta_S.L_i\}, \quad
\mu_{\beta_S}'^I := \max_i \{\beta_S.L'_i\}
\]
and
\[
\nu_{\beta_S}^I = \max \{ \mu_{\beta_S}^I  + \mu_{\beta_S}'^I, 0\} \ge 0.
\]
Define the \emph{$I$-minimal lift} of $\beta_S$ to be
$$
\beta_S^I := \beta_S - \mu_{\beta_S}^I \ell - \nu_{\beta_S}^I \gamma \in NE(X)
$$
where $\beta_S \in NE(X)$ is the \emph{canonical lift} such that
$h.\beta_S = 0 = \xi.\beta_S$.
\end{definition}

Clearly, $\beta_S^I$ is an effective class in $NE(X)$, as
$\mu^I_{\beta_S} \le \mu_{\beta_S}$ and $\nu^I_{\beta_S} \le \nu_{\beta_S}$.
When the inequality is strict,
the $I$-{minimal lift} is more effective than any geometric minimal lift.
Nevertheless it is uniquely defined and we will show that it encodes the 
information of the hypergeometric modification.

\begin{definition}
Define $\beta$ to be \emph{$I$-effective}, denoted $\beta \in NE^I(X)$, if
$$
 d \ge -\mu^I_{\beta_S} \quad \mbox{and} \quad d_2 \ge
 -\nu^I_{\beta_S}.
$$
It is called $\T I$-effective if $\beta$ is $I$-effective and $\T\beta$ is $I'$-effective. By the same proof of Proposition \ref{F-ef}, this is equivalent to
$$
d + \mu^I_{\beta_S} \ge 0 \quad \mbox{and} \quad d_2 - d + \mu_{\beta_S}'^I \ge 0.
$$
\end{definition}

\begin{lemma} [(Vanishing lemma)] \label{non-ef}
If $\bar\psi_* \beta \in NE(S)$ but $\beta \not\in NE(P)$ then
$I^{P/S}_\beta = 0$. In fact the vanishing statement holds for any $\beta = \beta_S + d\ell$ with $d < -\mu^I_{\beta_S}$.
\end{lemma}

\begin{proof}
We have $\beta.(h + L_i) = d + \mu_i \le d + \mu^I_{\beta_S} < 0$ for all $i$.
This implies that $I^{P/S}_\beta = 0$ since it contains the Chern
polynomial factor $\prod_i(h + L_i) = 0$ in the numerator.
\end{proof}

Now $I^{X/S}_\beta \equiv I^{Z/S}_\beta I^{X/Z}_\beta$ is given by
\begin{equation} \label{IX}
\prod_{i = 0}^r \frac{1}{\prod\limits_{m=1}^{\beta.a_i} (a_i + mz)} 
\prod_{i = 0}^r \frac{1}{\prod\limits_{m=1}^{\beta.b_i} (b_i + mz)} 
\frac{1}{\prod\limits_{m=1}^{\beta.\xi}
(\xi + mz)} =: A_\beta B_\beta C_\beta.
\end{equation}

Although (\ref{IX}) makes sense for any $\beta \in N_1(X)$, we have
\begin{lemma} \label{I-non-van}
$I^{X/S}_\beta$ is non-trivial only if $\beta \in NE^I(X)$.
\end{lemma}

\begin{proof}
Indeed, if $\beta_S \in NE(S)$ but $\beta \not\in NE^I(X)$ then either $d < -\mu^I_{\beta_S}$ and $A_\beta = 0$ by Lemma \ref{non-ef}, or $d \ge -\mu^I_{\beta_S}$ and we must have $d_2 < -\nu^I_{\beta_S} \le 0$ and all factors in $B_\beta$ appear in the numerator:
\begin{equation*}
d_2 - d + \mu'_i \le d_2 + \mu^I_{\beta_S} + \mu'^I_{\beta_S} \le d_2 + \nu^I_{\beta_S} < 0.
\end{equation*}
In particular $B_\beta C_\beta$ contains the Chern polynomial $f_{N \oplus \mathscr{O}} = 0$.
\end{proof}

\begin{remark}
In view of Lemma \ref{ef}, $\beta \in NE^I(X)$ is the ``effective condition for $\beta$ as if it is a primitive class''. One way to think about this is that the localization calculation of the $I$ factor is performed on the main component of the stable map moduli where $\beta$ is represented by a smooth rational curve.

As far as $I$ is concerned, the $I$-effective class plays the role of effective classes. However one needs to be careful that the converse of Lemma \ref{I-non-van} is not true: If $\beta$ is $I$-effective, it is still possible to have $I^{X/S}_\beta = 0$.
\end{remark}

The expression (\ref{IX}) agrees with (\ref{I-gamma}) by taking out the $z$ factor with $m$. The total factor is clearly
$$
z^{-(\sum_{i = 0}^r a_i + \sum_{i = 0}^{r + 1} b_i).\beta} = z^{-c_1(X/S).\beta}.
$$

Similarly for $\beta' \in NE(X')$, 
$I^{X'/S}_{\beta'} \equiv I^{Z'/S}_{\beta'} I^{X'/Z'}_{\beta'}$ is given by
\begin{equation} \label{IX'}
\prod_{i = 0}^r \frac{1}{\prod\limits_{m=1}^{\beta'.a'_i} (a'_i + mz)} \prod_{i = 0}^r
\frac{1}{\prod\limits_{m=1}^{\beta'.b'_i}
(b'_i + mz)} \frac{1}{\prod\limits_{m=1}^{\beta'.\xi'} (\xi' + mz)} =: A'_{\beta'} B'_{\beta'} C'_{\beta'}.
\end{equation}
Here $a'_i = h' + L'_i = \T b_i$ and $b'_i = \xi' - h' + L_i = \T a_i$.

By the invariance of the Poincar\'e pairing, $(\beta.a_i) = d + \mu_i = (\T\beta.b'_i)$ and $(\beta.b_i) = d_2 - d + \mu'_i =
(\T\beta.a'_i)$, and it is clear that all
the \emph{linear subfactors} in $I^{X/S}_\beta$ and $I^{X'/S}_{\T \beta}$ correspond
perfectly under $A_\beta \mapsto B'_{\T \beta}$,
$B_\beta \mapsto A'_{\T \beta}$ and $C_\beta \mapsto C'_{\T
\beta}$.

However, since the cup product is not preserved under $\T$, in
general $\T I_\beta \ne I'_{\T \beta}$. Clearly, any direct comparison of $I_\beta$ and $I'_{\T \beta}$ (without analytic continuations) can make sense only if $\beta$ is $\T I$-effective. This is the case for $(\beta.a_i)$'s (resp.~$(\beta.b_i)$'s) not all negative. Namely $A_\beta$ and $B_\beta$ both contain factors in the denominator.

\begin{lemma}[(Naive quasi-linearity)] \label{naive-ql}
\begin{itemize}
\item[(1)] $\T I_\beta.\xi = I'_{\T \beta}.\xi'$.

\item[(2)] If $d_2 := \beta.\xi < 0$ then $\T I_\beta = I'_{\T
\beta}$.
\end{itemize}
The expressions in {\rm (1)} or {\rm (2)} are nontrivial only if $\beta$
is $\T I$-effective.
\end{lemma}

\begin{proof}
(1) follows from the facts that $f: X \dasharrow X'$ is an isomorphism over the
infinity divisors $E \cong E$. For (2), notice that since $d_2 <
0$ the factor $C_\beta$ contains $\xi$ in the numerator
corresponding to $m = 0$. Similarly $C'_{\T \beta}$ contains
$\xi'$ in the numerator. Hence (2) follows from the same reason as
in (1). The last statement follows from Lemma \ref{I-non-van}.
\end{proof}

\subsection{Picard--Fuchs system}

Now we return to the BF/GMT constructed in Theorem \ref{BF/GMT} and multiply it by the infinity divisor $\xi$:
$$
J^X(\tau(\hat t)).\xi = P(z) I^X(\hat t).\xi.
$$
By Proposition \ref{red-GMT} and Lemma \ref{naive-ql}, we need to show the $\T$-invariance for $P(z)$ and $\tau(\hat t)$ in order to establish the general analytic continuation.

The very first evidence for this is that, as in the case of classical hypergeometric series, $I^X$ (resp.~$I^{X'}$) is a solution to certain Picard--Fuchs system which turns out to be $\T$-compatible:

\begin{proposition} [(Picard--Fuchs system on $X$)]\label{PF-eq-X}
$\Box_\ell I^X = 0$ and $\Box_\gamma I^X = 0$, where
$$
\Box_\ell = \prod_{j = 0}^r z \p_{a_j} - q^\ell e^{t^1} \prod_{j = 0}^r z\p_{b_j}, \qquad
\Box_\gamma = z\p_{\xi} \prod_{j = 0}^r z\p_{b_j} - q^\gamma e^{t^2}.
$$
\end{proposition}

Recall that $t^1$, $t^2$ are the dual coordinates of $h$, $\xi$ respectively. Here we use $\p_v$ to denote the directional derivative in $v$. Thus if $v = \sum v^i T_i \in H^2$ then $\p_v = \sum v^i \p_{t^i}$.

\begin{proof}
By extracting all the divisor variables $D = t^1 h + t^2 \xi$ and $\bar t_1 \in H^2(S)$ from $I^X$ (where $\bar t = \bar t_1 + \bar t_2$), we get
$$
I^X = \sum_{\beta \in NE(X)} q^\beta e^{\frac{D + \bar t_1}{z} + (D + \bar t_1).\beta} I^{X/S}_\beta J^S_{\beta_S}(\bar t_2).
$$
It is clear that $z\p_v$ produce the factor $v + z(v.\beta)$ for $v \in H^2$. From (\ref{IX}), $\prod_j z\p_{a_j}$ modifies the $A_\beta B_\beta C_\beta$ factor to
$$
\prod_{j = 0}^r\frac{1}{\displaystyle \prod_{m = 1}^{\beta.a_j - 1}(a_j + mz)}
B_\beta C_\beta = A_{\beta - \ell} B_{\beta -\ell} 
\prod_{j = 0}^r (b_j + z(\beta - \ell).b_j) C_{\beta - \ell}
$$
(since $\beta.a_j - 1 = (\beta - \ell).a_j$, $(\beta - \ell).b_j = \beta.b_j + 1$ and $(\beta - \ell).\xi = \beta.\xi$).

Clearly it equals the corresponding term from $q^\ell e^{t^1} \prod_j z\p_{b_j} I^X$ unless $\beta - \ell$ is not effective. But in that case the term is itself zero since $A_{\beta - \ell} = 0$ by Lemma \ref{non-ef}.

The proof for $\Box_\gamma I^X = 0$ is similar and is thus omitted.
\end{proof}

Similarly $I^{X'}$ is a solution to
$$
\Box_{\ell'} = \prod_{j = 0}^r z \p_{a'_j} - q^{\ell'} e^{-t^1} \prod_{j = 0}^r z\p_{b'_j}, \qquad
\Box_{\gamma'} = z\p_{\xi'} \prod_{j = 0}^r z\p_{b'_j} - q^{\gamma'} e^{t^2 + t^1},
$$
where the dual coordinates of $h'$ and $\xi'$ are $-t^1$ and $t^2 + t^1$ (since $\T (t^1 h + t^2 \xi) = t^1(\xi' - h') + t^2 \xi' = (-t^1)h' + (t^2 + t^1)\xi'$).

\begin{proposition} \label{p:6.11}
$$
\T \langle \Box^X_\ell, \Box^X_\gamma\rangle \cong
\langle \Box^{X'}_{\ell'}, \Box^{X'}_{\gamma'}\rangle.
$$
\end{proposition}

\begin{proof}
It is clear that
$$
\T \Box_\ell = -q^{-\ell'} e^{t^1} \Box_{\ell'},
$$
and
\begin{equation*}
\T \Box_\gamma = z\p_{\xi'} \prod_{j = 0}^r z\p_{a'_j} - q^{\gamma' + \ell'} e^{t^2} = z\p_{\xi'} \Box_{\ell'} + q^{\ell'} e^{-t^1} \Box_{\gamma'}.
\end{equation*}
\end{proof}

Namely, the Picard--Fuchs system on $X$ and $X'$ are indeed equivalent under $\T$. Moreover, both $I = I^X$ and $I ' = I^{X'}$ satisfy this system, but in different coordinate charts ``$|q^\ell| < 1$'' and ``$|q^{\ell}| > 1$'' (of the K\"ahler moduli) respectively.

We do not expect $I$ and $I'$ to be the same solution under analytic 
continuations in general. 
In fact, they are not in some examples. 
We know this is not true for $J$ and $J'$ since the general descendent 
invariants are not $\T$-invariant. 
Nevertheless it turns out that $P(z)$ and $\tau(\hat t)$ are correct objects 
to admit $\T$-invariance.

\begin{lemma} \label{B=1}
Modulo $q^{\beta_S}$, $\beta_S \in NE(S)$ and $\gamma$, we have $P(z) \equiv 1$ and $\tau(\hat t) \equiv \hat t$.
\end{lemma}

\begin{proof}
One simply notices that in the proof of Theorem \ref{BF/GMT} to construct $P(z)$, the induction can be performed on $[\beta] = (\beta_S, d_2) \in W $, 
as in Part I Section~3.2, by removing the whole series in $q^\ell$ with the same top non-negative $z$ power once a time. For the initial step $[\beta] = 0$ and $J^S([\beta] = 0) = e^{\bar t/z}$, from (\ref{IX}) we have extremal ray contributions:
$$
I_{[\beta] = 0} = e^{\hat t/z} (1 + O(1/z^{r + 1})).
$$
As there is no non-negative $z$ powers besides 1, also later inductive steps will create only higher order $q^{[\beta]}$'s with respect to $W$, hence the result follows.
\end{proof}

\begin{remark} \label{homg-deg}
By the virtual dimension count and (\ref{e:5.1}), $J$ is weighted homogeneous of degree 0 in the following weights $|\cdot|$: We set $|T_\mu|$ to be its Chow degree, $|t^\mu| = 1 - |T_\mu|$, $|q^\beta| = (c_1(X).\beta)$ and $|\psi| = |z| = 1$.
This is usually expressed as: The Frobenius manifold $(QH(X), *)$ is conformal with respect to the Euler vector field
$$
E = \sum (1 - |T_\mu|) t^\mu \p_\mu + c_1(X) \in \Gamma(TH).
$$
For the hypergeometric modification $I$, the base $J^S$ has degree 0 with $|q^{\beta_S}| = (c_1(S).\beta_S)$. But when $\beta_S$ is viewed as an object in $X$ the weight increases by $(c_1(X/S).\beta_S)$. This cancels with the weight of the factor $I^{X/S} q^{\beta - \beta_S}$, which is
\[
 \begin{split}
 &\quad -c_1(X/S).\beta + c_1(X).\beta - c_1(X).\beta_S \\
 &= c_1(S).\beta - c_1(X).\beta_S \\
 &= -c_1(X/S).\beta_S.
 \end{split}
\]
Hence \emph{$I$ is also homogeneous of degree 0}.
\end{remark}

\section{Extension of quantum $\D$ modules via quantum Leray--Hirsch}
In this section we will complete the proof of the main theorem
(Theorem \ref{main-thm}) on invariance of quantum rings under
ordinary flops of splitting type.
Proposition~\ref{p:6.11} guarantees the $\T$-invariance of the
Picard--Fuchs systems (in the fiber directions).
In order to construct the $\D$ module $\mathscr{M}_I = \D I$,
we will need to find the derivatives in the general base directions.
This will be accomplished by a lifting of the QDE on the base $S$.
Putting these together, we will show that they generate enough (correct) 
equations for $\mathscr{M}^X_I$.
This is referred as the quantum Leray--Hirsch theorem,
which is the content of Theorem \ref{thm:0.5}
($=$ Theorem \ref{I-lift} + Theorem \ref{q-LH} + Theorem \ref{Naturality}).

To obtain the (true) quantum $\D$-module $\mathscr{M}_J^X$
(on a sufficiently large Zariski closed subset given by the image of
$\tau(\hat t)$), we apply the Birkhoff factorization on $\mathscr{M}_I^X$.
We specifically choose a way to perform BF such that
the $\T$-invariance can be checked more naturally.

Before proceeding to the first step, let us lay out the notations
and conventions for this section.

\begin{notations} \label{n:1}
We use $\bar\beta \in NE(S)$, $\bar t \in H(S)$ etc.\ to denote objects in $S$.
When they are viewed as objects in $X$, $\bar\beta$ means the canonical lift,
$\bar t$ means the pullback $\bar p^*: H(S) \to H(X)$.

For a basis $\{\bar T_i\}$ of $H(S)$, denote
$\bar t = \sum \bar t^i \bar T_i$ a general element in $H(S)$.
When $\bar T_i$ is considered as an element in $H(X)$,
we sometimes abuse the notation by setting  $T_i := \bar T_i$.

Given a basis $\{\bar T_i\}$ of $H(S)$, we use the following
\emph{canonical basis} for $H(X)$:
$$\{T_\e= \bar T_i h^l \xi^m \mid 0\le l \le r, 0\le m \le r + 1\}.$$
A general element in $H(X)$ is denoted $t = \sum t^\e T_\e$.
The index set of the canonical basis is denoted $\Lambda^+$.

By abusing the notations, if $T_\e = \bar{T}_i$ (i.e.~$l=m=0$), we set
$t^\e = t^i = \bar t^i$.
Similarly we set $t^\e = t^1$ for $T_\e = h$, and $t^\e = t^2$ for $T_\e = \xi$.
That is, we reserve the index 0, 1 and 2 for $1$, $h$ and $\xi$ respectively.

On $H(X')$ the canonical basis is chosen to be
$$\{ T'_\e := \T T_\e = \bar T_i (\xi ' - h')^l \xi'^m \}$$
so that it shares the same coordinate system as $H(X)$:
$$
t = \sum_\e t^\e T_\e \mapsto \T t = \sum_\e t^\e \T T_\e = \sum_\e t^\e T'_\e.
$$
\end{notations}

\subsection{$I$-lifting of the Dubrovin connection}

Let the quantum differential equation of $QH(S)$ be given by
$$
z\p_i z\p_j J^S(\bar t) = \sum_k \bar C_{ij}^k (\bar t, \bar q)\,z\p_k J^S(\bar t).
$$
If we write $\bar C_{ij}^k(\bar t, \bar q) = \sum \bar C_{ij,\bar\beta}^k(\bar t)\, q^{\bar \beta}$, then the effect on the $\bar{\beta}$-components reads as
$$
z\p_i z\p_j J^S_{\bar \beta} = \sum_{k, \bar\beta_1} \bar C_{ij, \bar\beta_1}^k \,z\p_k J^S_{\bar\beta - \bar\beta_1}.
$$

Now we lift the equation to $X$. In the following, for a curve class $\bar \beta \in NE(S)$, its $I$-minimal lift in $NE(X)$ is denoted by $\bar\beta^I$.  We compute
\begin{equation} \label{lift-1}
\begin{split}
z\p_i z\p_j I &= \sum_{\beta} q^\beta e^{\frac{D}{z} + (D.\beta)} I^{X/S}_{\beta} z\p_i z\p_j J^S_{\bar \beta}\\
&= \sum_{k, \beta, \bar\beta_1} q^\beta e^{\frac{D}{z} + (D.\beta)} I^{X/S}_{\beta}  \bar C_{ij, \bar\beta_1}^k\,z\p_k J^S_{\bar\beta - \bar\beta_1} \\
&= \sum_{k, \bar\beta_1} q^{\bar\beta_1^I} e^{D.\bar\beta_1^I} \bar C_{ij, \bar\beta_1}^k z\p_k \sum_\beta q^{\beta - \bar\beta_1^I} e^{\frac{D}{z} + D.(\beta - \bar\beta_1^I)} I^{X/S}_\beta J^S_{\bar\beta - \bar\beta_1}.
\end{split}
\end{equation}
The terms in last sum are non-trivial only if
$\bar\beta - \bar\beta_1 \in NE(S)$.
However, in this presentation it is not a priori guaranteed that
$\beta - \bar\beta_1^I$ is $I$-effective.
(Hence, there might be some vanishing terms in the presentation.)

In order to obtain the RHS as an operator acting on $I$,
we will seek to ``transform'' terms of the form
$e^{\frac{D}{z} + D.(\beta - \bar\beta_1^I)} I^{X/S}_\beta J^S_{\bar\beta - \bar\beta_1}$
to those of the form
$e^{\frac{D}{z} + D.(\beta - \bar\beta_1^I)} I^{X/S}_{\beta - \bar\beta_1^I} J^S_{\bar\beta - \bar\beta_1}$.
This can be achieved by differentiation the RHS judiciously
and will be explained below.

As a first step, we will show that $I^{X/S}_{\beta} = 0$
if $\beta - \bar\beta_1^I \not\in NE^I(X)$ and
$\bar\beta - \bar\beta_1 \in NE(S)$.

\begin{definition}
For any one cycle $\beta \in A_1(X)$, effective or not,
we define
\[
 \begin{split}
 n_i(\beta) &:= -\beta.(h + L_i), \\
 n'_i(\beta) &:= -\beta.(\xi - h + L'_i), \\
 n'_{r + 1}(\beta) &:= -\beta.\xi,
 \end{split}
\]
where $0 \le i \le r$.
\end{definition}

\begin{lemma}
For $\bar\beta \in NE(S)$, the $I$-minimal lift $\bar\beta^I \in NE(X)$
satisfies $ n_i(\bar\beta^I) \ge 0$, $ n'_i(\bar\beta^I)\ge 0$ for all $i$.
\end{lemma}

\begin{proof}
During the proof, the \emph{superscript} $I$ is omitted for simplicity.

By definition,
$$ n_i= - \bar\beta^I.(h + L_i) = \mu - \mu_i \ge 0.$$
Similarly for $0 \le i \le r$,
$$
 n'_i = -\bar\beta^I.(\xi - h + L_i') =
 \max\{\mu + \mu', 0\} - \mu - \mu'_i.
$$
If $\mu + \mu' \ge 0$, we have
$$ n'_i = \mu' - \mu'_i \ge 0.$$
Otherwise if $\mu + \mu' < 0$, then we get
\begin{equation} \label{mu-gap}
 n'_i = 0 - (\mu + \mu'_i) \ge - (\mu + \mu') > 0.
\end{equation}
Finally for the compactification factor $\mathscr{O}$, we get
$$
 n'_{r+1} = - \bar\beta^I.\xi = \max \{\mu + \mu', 0\} \ge 0.
$$
\end{proof}

Let $\beta, \beta' \in A_1(X)$ be (not necessarily effective) one cycles.
By definition of $I$-function, the $\beta$ factor corresponding to $h + L_i$ is
\begin{equation*}
A_{i, \beta} = \frac{1}{\prod\limits_{m =
1}^{\beta.(h + L_i)} (h + L_i + mz)},
\end{equation*}
which depends only on the intersection number.
Suppose that
\[
 l_i:= \beta'.(h + L_i) - \beta.(h + L_i) \ge 0,
\]
we have
\begin{equation} \label{prod-comp}
 A_{i, \beta} =
 A_{i, \beta'} \prod_{m = \beta.(h + L_i) + 1}^{\beta'.(h + L_i)} (h + L_i + mz).
\end{equation}
We say that $A_{i, \beta}$ is a product of $A_{i, \beta'}$ with a
(cohomology-valued) factor of length $l_i$.
The factors corresponding to
$\xi - h + L_i'$ and $\xi$ behave similarly.

\begin{lemma}
Let $\beta \in NE(X)$ and
$\beta - \bar\beta_1^I$ be an $I$-effective class.
$I^{X/S}_\beta$ is the product of $I^{X/S}_{\beta - \bar\beta^I_1}$ with a factor
which is a product of length $n_i(\bar\beta^I_1)$,
$n_i'(\bar\beta^I_1)$, and $n'_{r + 1}(\bar\beta^I_1)$ corresponding to
$h + L_i$, $\xi - h + L_i'$, and $\xi$ respectively.

If $\beta - \bar\beta_1^I$ is not $I$-effective,
the conclusion holds in the sense that $I^{X/S}_\beta = 0$.
\end{lemma}

\begin{proof}
Set $\beta' = \beta - \bar\beta_1^I$ in \eqref{prod-comp},
the length is
$$(\beta' -\beta).(h + L_i) = -\bar\beta_1^I.(h + L_i) = n_i(\bar\beta^I_1).$$
The argument for $\xi - h + L_i'$ and $\xi$ are similar.

If $\beta - \bar\beta_1^I$ is not $I$-effective,
formally $I^{X/S}_{\beta - \bar\beta_1^I} = 0$ contains either the Chern polynomial
$f_F$ or $f_{N \oplus \mathscr{O}}$ in its numerator.
Notice that (\ref{prod-comp}) holds formally.

This proves the lemma.
\end{proof}

our next step is to show that the factors in (\ref{prod-comp})
can be obtained by introducing certain differential operators acting on $I$.

\begin{definition} \label{D-op}
An one cycle $\beta \in A_1(X)$ is called \emph{admissible} if $n_i(\beta) \ge 0$, $n'_i(\beta) \ge 0$, and $n'_{r + 1}(\beta) \ge 0$. For admissible $\beta$ we define differential operators
\begin{equation*}
\begin{split}
D^A_{\beta} &:= \prod_{i = 0}^r\prod_{m = 0}^{n_i(\beta) - 1} (z\p_{h + L_i} - mz), \\
D^B_{\beta} &:= \prod_{i = 0}^r\prod_{m = 0}^{n'_i(\beta) - 1} (z\p_{\xi - h + L'_i} - mz), \\
D^C_{\beta} &:= \prod_{m = 0}^{n'_{r + 1}(\beta) - 1} (z\p_\xi - mz), \\
D_{\beta}(z) &:= D^A_{\beta}D^B_{\beta}D^C_{\beta}.
\end{split}
\end{equation*}
\end{definition}

Now we are ready to lift the quantum differential equations for $J^S$ to
equations for $I^X$.

\begin{theorem} [($I$-lifting of QDE)] \label{I-lift}
The Dubrovin connection on $QH(S)$ can be lifted to $H(X)$ as
\begin{equation} \label{Dubrovin-lift}
z\p_i z\p_j I = \sum_{k, \bar\beta} q^{\bar\beta^*} e^{D.\bar\beta^*}\bar C_{ij, \bar\beta}^k(\bar t)\, z\p_k D_{\bar\beta^*}(z) I
\end{equation}
where $\bar\beta^* \in A_1(X)$ is any admissible lift of $\bar\beta$,
which in particular implies the well-definedness of the operators
$D_{\bar\beta^*}(z)$.

Furthermore, one can always choose $\bar\beta^*$ to be effective.
An example of an effective lift is the $I$-minimal lift
$\bar\beta^* = \bar\beta^I$,
which is the only admissible lift if and only if $\mu + \mu' \ge 0$.

In general, all liftings are related to each other modulo the
Picard--Fuchs system generated by $\Box_\ell$ and $\Box_\gamma$.
\end{theorem}

\begin{proof}
We apply the calculation in (\ref{lift-1}) with $\bar\beta^I_1$ being replaced by a general admissible lift $\bar\beta^*_1$. For $\bar t = \bar t_1 + \bar t_2$ with $\bar t_1$ being the divisor part,
\begin{align*}
&\sum_\beta q^{\beta - \bar\beta_1^*} e^{\frac{D}{z} + D.(\beta - \bar\beta_1^*)} I^{X/S}_\beta J^S_{\bar\beta - \bar\beta_1}(\bar t) \\
& \qquad = \sum_\beta D_{\bar\beta^*_1}(z) q^{\beta - \bar\beta_1^*} e^{\frac{D + \bar t_1}{z} + (D + \bar t_1).(\beta - \bar\beta_1^*)} I^{X/S}_{\beta - \bar\beta_1^*} J^S_{\bar\beta - \bar\beta_1}(\bar t_2) = D_{\bar\beta^*_1}(z) I.
\end{align*}

Now we prove the last statement. Any two (admissible) lifts differ by some $a\ell + b\gamma$. Say, $\beta'' = \beta' + a\ell + b\gamma$. Then we have
\begin{equation} \label{adm-step}
\begin{split}
n_i(\beta'') &= n_i(\beta') - a, \\
n'_i(\beta'') &= n'_i(\beta') + (a - b), \\
n'_{r + 1}(\beta'') &= n'_{r + 1}(\beta') - b.
\end{split}
\end{equation}
Then it is elementary to see that we may connect $\beta'$ to $\beta''$ by adding or subtracting $\ell$ or $\gamma$ once a time, with all the intermediate steps $\beta'_j$ being admissible. For example, if $a > 0$, $b > 0$ and $a - b > 0$, then we start by adding $\ell$ up to $j = a - b$ times. Then we iterate the process: Adding $\gamma$ followed by adding $\ell$, up to $b$ times. Thus we only have to consider the two cases (1) $\beta'' = \beta' + \ell$ or (2) $\beta'' = \beta' + \gamma$.

For case (1), we get from (\ref{adm-step}) with $(a, b) = (1, 0)$ that $n_i(\beta') \ge 1$ for all $i$. This implies that $D^A_{\beta'} = D^{A+}_{\beta'} D^A_0$ where $D^A_0 = \prod_{j = 0}^r z\p_{a_j}$ comes from the product of $m = 0$ terms. Since $\Box_\ell I = 0$, we compute
$$
D_{\beta'}(z) I = D^B_{\beta'} D^C_{\beta'} D^{A+}_{\beta'} q^\ell e^{t^1} \prod_{j = 0}^r z\p_{b_j} I.
$$
Now we move $q^\ell e^{t^1}$ to the left hand side of all operators by noticing
$$
z\p_h e^{t^1} = e^{t^1} (z\p_h + z)
$$
in the operator sense. Then (notice that $D^C_{\beta'} = D^C_{\beta' + \ell}$)
$$
D_{\beta'}(z) I = q^\ell e^{t^1} D^{B+}_{\beta' + \ell} D^C_{\beta'} D^A_{\beta' + \ell} \prod_{j = 0}^r z\p_{b_j} I = q^\ell e^{t^1} D_{\beta' + \ell}(z) I,
$$
which is the desired factor for $\beta''$.

The proof for case (2) is entirely similar, with $\Box_\gamma I = 0$ being used instead, and is thus omitted.

The uniqueness statement for $\mu + \mu' \ge 0$ follows from (\ref{adm-step}) and the observation: $n_i(\bar\beta^I) = \mu - \mu_i$ and $n'_i(\bar\beta^I) = \mu' - \mu'_i$, both attain 0 somewhere and there is no room to move around.
The proof is complete.
\end{proof}

Notice that the liftings of QDE may not be unique. We will see the importance of such a freedom when we discuss the $\T$-invariance property.

\subsection{Quantum Leray--Hirsch}

\begin{definition} \label{d:7.7}
Let $T_\e = \bar{T}_i h^l \xi^m$ be an element in the canonical basis of $H(X)$.
The \emph{naive quantization} of $T_\e$ is defined as (c.f.~(\ref{zp}) and (\ref{naive-q}))
\[
\hat T_\e := \p^{z\e} = z \p_{\bar{t}^i} (z \p_{t^1})^l (z \p_{t^2})^m.
\]
\end{definition}

\begin{theorem} [(Quantum Leray--Hirsch)] \label{q-LH}
The $I$-lifting (\ref{Dubrovin-lift}) of quantum differential equations on $S$
and the Picard--Fuchs equations determine a first order matrix system under
the naive quantization $\p^{z\e}$ of canonical basis $T_\e$'s of $H(X)$:
$$
z\p_a (\p^{z\e} I) = (\p^{z\e} I) C_a(z, q), \qquad t^a \in \{t^1, t^2, \bar t^i\}.
$$

This system has the property that for any fixed $\bar\beta \in NE(S)$,
the coefficients are formal functions in $\bar t$ and polynomial functions
in $q^\gamma e^{t^2}$, $q^\ell e^{t^1}$ and the basic rational
function $\f(q^\ell e^{t^1})$, defined in (\ref{rat-f}).
\end{theorem}

We start with an overview of the general ideas involved in the proof.
The Picard--Fuchs system generated by $\Box_\ell$ and $\Box_\gamma$ is a
perturbation of the Picard--Fuchs (hypergeometric) system associated
to the (toric) fiber by operators in base divisors.
The fiberwise toric case is a GKZ system,
which by the theorem of Gelfand--Kapranov--Zelevinsky is a holonomic system
of rank $(r + 1)(r + 2)$, the dimension of cohomology space of a fiber.
It is also known that the GKZ system admits a Gr\"obner basis reduction
to the holonomic system.

We apply this result in the following manner:
We will construct a $\D$ module with basis $\p^{z\e}$, $\e \in \Lambda^+$.
We apply operators $z\p_{t^1}$, $z\p_{t^2}$ and first order operators $z\p_i$'s
to this selected basis. Notice that
\begin{align*}
\Box_\ell &= (1 - (-1)^{r + 1}q^\ell e^{t^1}) (z\p_{t^1})^{r + 1} + \cdots, \\
\Box_\gamma &= (z\p_{t^2})^{r + 2} + \cdots.
\end{align*}
The Gr\"obner basis reduction allows one to reduce the differentiation order
in $z\p_{t^1}$ and $z\p_{t^2}$ to smaller one.
In the process higher order differentiation in $z\p_i$'s will be introduced.
Using the $I$-lifting, the differentiation in the base direction with order
higher than one can be reduced to one by introducing more terms with strictly
larger effective classes in $NE(S)$.
A refinement of these observations will lead to a proof,
which is presented below.

\begin{remark}
In fact, neither the Gr\"obner basis nor the GKZ theorem will be needed,
due to the simple feature of the Picard--Fuchs system we have for split
ordinary flops.
\end{remark}

\begin{proof}
Consider first the case of simple $P^r$ flops ($S = {\rm pt}$).
In this special case the Gr\"obner basis is already at hand.
The naive quantization of canonical cohomology basis gives
$$
 \p^{z(i, j)} := (z\p_{t^1})^i (z\p_{t^2})^j, \qquad 0 \le i \le r,
 \quad 0 \le j \le r  +1.
$$
Then further differentiation in the $t^1$ direction leads to
$$
z\p_{t^1} \p^{z(i, j)} = \p^{z(i + 1, j)}.
$$
It is clear that we only need to deal with the boundary case $i = r$,
when the RHS goes beyond the standard basis.

\textbf{Case $(i,j)=(r,0)$}.
The equation $\Box_\ell = (z\p_{t^1})^{r + 1} - q^\ell e^{t^1}(z\p_{t^2} - z\p_{t^1})^{r + 1} \equiv 0$ modulo $I$ leads to
\begin{equation} \label{hr+1}
(z\p_{t^1})^{r + 1} \equiv \frac{q^\ell e^{t^1}}{1 - (-1)^{r + 1}q^\ell e^{t^1}} \sum_{k = 1}^{r + 1} C^{r + 1}_k (z\p_{t^2})^k (-z\p_{t^1})^{r + 1 - k},
\end{equation}
which solves the case.

\textbf{Case $(i,j)=(r, \ge 1)$}.
For $j \ge 1$, notice that $\Box_\gamma = z\p_{t^2} (z\p_{t^2} - z\p_{t^1})^{r + 1} - q^\gamma e^{t^2}\equiv 0$ modulo $I$. Hence
\begin{equation} \label{close-up}
\begin{split}
(z\p_{t^1})^{r + 1} (z\p_{t^2})^j &= q^\ell e^{t^1} (z\p_{t^2})^j (z\p_{t^2} - z\p_{t^1})^{r + 1} \\
&\equiv q^\ell e^{t^1} (z\p_{t^2})^{j - 1} q^\gamma e^{t^2}\\
&= q^\ell e^{t^1} q^\gamma e^{t^2} (z\p_{t^2} + z)^{j - 1}.
\end{split}
\end{equation}
This in particular solves the other cases $1 \le j \le r + 1$.

Similarly differentiation in the $t^2$ direction:
$$
z\p_{t^2} \p^{z(i, j)} = \p^{z(i, j + 1)}.
$$
And we only need to deal with the boundary case $j = r + 1$.

\textbf{Case $(i,j)=(0, r+1)$}.
First of all, $\Box_\gamma I = 0$ leads to
\begin{equation} \label{Box-gamma-relation}
\begin{split}
&(z\p_{t^2})^{r + 2} \\
&\equiv -(-1)^{r + 1}(z\p_{t^1})^{r + 1} z\p_{t^2} - \sum_{k = 1}^r C^{r + 1}_k (z\p_{t^2})^{k + 1} (-z\p_{t^1})^{r + 1 - k} + q^\gamma e^{t^2}\\
&= (1 - (-1)^{r + 1}q^\ell e^{t^1}) q^\gamma e^{t^2} - \sum_{k = 1}^r (-1)^{r + 1 - k} C^{r + 1}_k \p^{z(r + 1 - k, k + 1)},
\end{split}
\end{equation}
which solves the case.

\textbf{Case $(i,j)=(\ge 1, r+1)$}.
By further differentiating $t^1$ on  (\ref{Box-gamma-relation}) and on (\ref{close-up}),  we get
\begin{equation} \label{xir+2+}
\begin{split}
&(z\p_{t^1})^i (z\p_{t^2})^{r + 2} \\
&\equiv (z\p_{t^1})^i q^\gamma e^{t^2} - (-1)^{r + 1} (z\p_{t^1})^i q^\ell e^{t^1}q^\gamma e^{t^2}\\
& \qquad - \sum_{k = 1}^r (-1)^{r + 1 - k} C^{r + 1}_k (z\p_{t^1})^{r + 1 + (i - k)} (z\p_{t^2})^{k + 1}\\
&= q^\gamma e^{t^2} (z\p_{t^1})^i - (-1)^{r + 1} q^\ell e^{t^1}q^\gamma e^{t^2} (z\p_{t^1} + z)^i \\
& \qquad -\sum_{k = i + 1}^{r}(-1)^{r + 1 - k} C^{r + 1}_k \p^{z(r + i + 1 - k, k + 1)} \\
& \qquad - q^\ell e^{t^1}q^\gamma e^{t^2}\sum_{k = 1}^i (-1)^{r + 1 - k} C^{r + 1}_k (z\p_{t^1} + z)^{i - k} (z\p_{t^2} + z)^{k}.
\end{split}
\end{equation}
This in particular solves the remaining cases $1 \le i \le r$.

An important observation of the above calculation of the matrix
$C_1(z, q)$, $C_2(z, q)$ is that $C_i$ is constant in $z$ when modulo $q^\gamma$.
Moreover $q^{d_2 \gamma}$ appears only in $d_2 = 1$.

Now we consider the case with base $S$. The Picard--Fuchs equations are
\begin{equation} \label{PF-eq}
\begin{split}
\Box_\ell &= \prod_{j = 0}^r z \p_{h + L_j} - q^\ell e^{t^1} \prod_{j = 0}^r z\p_{\xi - h + L'_j}, \\
\Box_\gamma &= z\p_{\xi} \prod_{j = 0}^r z\p_{\xi - h + L'_j} - q^\gamma e^{t^2}.
\end{split}
\end{equation}
Recall that for a basis element $T_\e = \bar T_s h^i \xi ^j$ in its canonical presentation ($0 \le i \le r$, $0 \le j \le r + 1$), we associated its naive quantization
\begin{equation} \label{q-basis}
\hat T_\e = \p^{z\e} = z\p_{\bar t^s} (z\p_{t^1})^i (z\p_{t^2})^j.
\end{equation}
The above calculations (\ref{hr+1}) --- (\ref{xir+2+}) need to be corrected by adding more differential symbols which may consist of higher derivatives in base divisors $z\p_{L_j}$'s and $z\p_{L'_j}$'s instead of a single $z\p_{\bar t^s}$. Thus they are not yet in the desired form (\ref{q-basis}). The $I$-lifting (\ref{Dubrovin-lift}) helps to reduce higher derivatives in base to the first order ones. Although new derivatives $D_{\bar\beta}$'s may appear during this reduction, it is crucial to notice that they all come with non-trivial classes $q^{\bar\beta^I}$'s.

With these preparations, we will prove the theorem by constructing
$$
C_{a, \bar\beta}(z) = \sum_{\beta \mapsto \bar\beta} C_{a; \beta}(z)\, q^\beta
$$
for any fixed $\bar\beta \in NE(S)$.

For $\bar\beta = 0$, the $I$-lifting (\ref{Dubrovin-lift}) introduces no further derivatives: $D_{\bar\beta = 0}(z) = {\rm Id}$. Thus higher order differentiations on $\bar t^s$'s can all be reduced to the first order. Notice that in (\ref{PF-eq}) all the corrected terms have $(z\p_{t^1})^i (z\p_{t^2})^j$ in the canonical range, hence (\ref{hr+1}) --- (\ref{xir+2+}) plus (\ref{Dubrovin-lift}) lead to the desired matrix $C_{a; \bar\beta = 0}(z)$.

Given $\bar\beta \in NE(S)$, to determine the coefficient $C_{a, \bar\beta}$ from calculating $z\p_a (\p^{z\e})$, it is enough to consider the restriction of (\ref{Dubrovin-lift}) to the finite sum over $\bar\beta' \le \bar\beta$. We repeatedly apply the following two constructions:

(i) The double derivative in base can be reduced to single derivative by (\ref{Dubrovin-lift}). If a new non-trivial derivative $D_{\bar\beta_1^I}(z)$ is introduced then we get a new higher order term with respect to $NE(S)$ since the factor $q^{\bar\beta^I_1}$ is added, thus such processes will produce classes with image outside $NE_{\le \bar\beta}(S)$ in finite steps. In fact the only term in (\ref{Dubrovin-lift}) not increasing the order in $NE(S)$ is given by
$$
\bar C_{aj; \bar\beta = 0}^k\, z\p_k.
$$
This is precisely the structural constant of cup product on $H(S)$, which is non-zero only if
$$
\deg \bar T_a + \deg \bar T_j = \deg \bar T_k.
$$
Hence $\deg \bar T_k \ge \deg \bar T_a$, with equality holds only if $\bar T_j = 1$, which may occur only for the first step. Any further reduction of base double derivatives $z\p_k z\p_l$ into a single derivative $z\p_m$ must then increase the cohomology degree $\deg \bar T_m > \deg \bar T_k$, if the order in $NE(S)$ is not increased. It is clear the process stops in finite steps.

(ii) Each time we have terms not in the reduced form (\ref{q-basis}) we perform the Picard--Fuchs reduction (\ref{hr+1}) --- (\ref{xir+2+}) with correction terms. After the first step in simplifying $z\p_{t^1} (\p^{z\e})$ and $z\p_{t^2} (\p^{z\e})$, in all the remaining steps we face such a situation only when we have non-trivial terms $D_{\bar\beta_1^I}(z)$ from construction (i). As before this produces classes with image outside $NE_{\le \bar\beta}(S)$ in finite steps.

Combining (i) and (ii) we obtain $C_{a; \bar\beta}$ in finite steps. It is clearly polynomial in $z$, $q^\gamma e^{t^2}$, $q^\ell e^{t^1}$ and $\f(q^\ell e^{t^1})$ since this holds for each steps.
\end{proof}

\begin{theorem} [(Naturality)] \label{Naturality}
The system is $\T$-invariant. That is, $\T C_a (\hat t) \cong C'_a (\T \hat t)$.
\end{theorem}

\begin{proof}
We have seen the $\T$-invariance of the Picard--Fuchs systems. It remains to show the $\T$-invariance of the $I$-lifting of the base Dubrovin connection, up to modifications by $\Box_\ell$ and $\Box_\gamma$.

By (\ref{Dubrovin-lift}), the simplest situation to achieve such an invariance is the case that $\T \bar\beta^I = \bar\beta^{I'}$, since then $\T D_{\bar\beta^I}(z) = D'_{\bar\beta^{I'}}(z)$ as well.

Indeed, when $\mu + \mu' \ge 0$ for a curve class $\bar\beta$, we do have
\begin{align*}
\T \bar\beta^I &= \T(\bar\beta -\mu \ell - (\mu + \mu')\gamma) \\
&= \bar\beta + \mu \ell' - (\mu + \mu')(\ell' + \gamma') \\
&= \bar\beta - \mu' \ell' - (\mu + \mu') \gamma' = \bar\beta^{I'}.
\end{align*}

It remains to analyze the case $\mu + \mu' < 0$ for $\bar\beta$. In this case,
$$
\T \bar\beta^I - \bar\beta^{I'}= \bar\beta + \mu \ell' - (\bar\beta - \mu'\ell') = (\mu + \mu') \ell' = -\delta \ell',
$$
where $\delta := -(\mu + \mu') > 0$ is the finite gap. Thus
$$
\T q^{\bar\beta^I - \delta \ell} = q^{\bar\beta^{I'}}
$$
and this suggests that we should try to decrease $\bar\beta^I$ by $\ell$ for $\delta$ times.

In other words, we should expect to have another valid lifting:
\begin{equation} \label{twisted-lift}
z\p_i z\p_j I = \sum_{k, \bar\beta} q^{\bar\beta^I - \delta\ell} e^{D.(\bar\beta^I-\delta\ell)}\bar C_{ij, \bar\beta}^k(\bar t)\, z\p_k D_{\bar\beta^I - \delta\ell}(z) I.
\end{equation}
This is easy to check: Notice that $n_i(\bar\beta^I - \delta \ell) = n_i(\bar\beta^I) + \delta > 0$. $n'_i(\bar\beta^I - \delta) = n'_i(\bar\beta^I) - \delta$, which is also $n'_i(\bar\beta + \mu'\ell) = \mu' - \mu'_i \ge 0$ (c.f.~the gap in (\ref{mu-gap})). $n'_{r + 1} \ge 0$ is unchanged. Thus, the operator $D_{\bar\beta^I -\delta \ell}$ is well defined, though $\bar\beta^I - \delta \ell$ may not be effective. By Theorem \ref{I-lift}, (\ref{twisted-lift}) is also a lift and the theorem is proved.
\end{proof}

\subsection{Reduction to the canonical form: The final proof}

We will construct a gauge transformation $B$ to eliminate all the $z$
dependence of $C_a$.
The final system is then equivalent to the Dubrovin connection on $QH(X)$.
Such a procedure is well known in small quantum cohomology of Fano type
examples or in the context of abstract quantum cohomology.
(See e.g.~\cite{Guest} and references therein.)
Here we will also need to take into account the role played by the
generalized mirror transformation (GMT) $\tau(\hat t)$.

In fact $B$ is nothing more than the Birkhoff factorization introduced before:
\begin{equation} \label{bf}
(\p^{z\e} I (\hat t)) = (z\nabla J) (\tau) B(\tau)
\end{equation}
valid at the generalized mirror point $\tau = \tau(\hat t)$.
Thus $B$ exists uniquely via an inductive procedure.
However the analytic (formal) dependence of $B$ is not manifest if one
proceeds in this direction,
as the procedure involves $J$ and $I$, for neither the analytic
dependence holds.
Therefore, it is not clear how to prove $\T B \cong B'$
up to analytic continuations.

In this subsection we will proceed in a slightly different, but ultimately
equivalent, way.
Namely we study instead the gauge transformation $B$ directly from the
differential system
\begin{equation} \label{system-1}
z\p_a (\p^{z\e} I) = (\p^{z\e} I) C_a.
\end{equation}
Even though the solutions $I$ are not $\T$-invariant,
the system is $\T$-invariant by Theorem \ref{Naturality}.
This system can be analyzed inductively with respect to the partial
ordering of the Mori cone on the base $NE(S)$.

Substituting (\ref{bf}) into (\ref{system-1}), we get $z\p_a (\nabla J) B + z (\nabla J) \p_a B = (\nabla J) B C_a$, hence
\begin{equation} \label{tilde-Ca}
z\p_a (\nabla J) = (\nabla J) (-z \p_a B + BC_a)B^{-1} =: (\nabla J) \tilde C_a.
\end{equation}

We note the subtlety in the meaning of $\tilde C_a (\hat t)$.
Let $\tau = \sum \tau^\mu T_\mu$.
Then the QDE reads as
$$z\p_\mu (\nabla J)(\tau) = (\nabla J) (\tau) \tilde C_\mu (\tau),$$
where $\tilde C_\mu (\tau)$ is the structure matrix of quantum multiplication
at the point $\tau \in H(X)$. Then
\begin{equation*}
z\p_a (\nabla J) = \sum_\mu \frac{\p \tau^\mu}{\p t^a} z\p_\mu (\nabla J) = (\nabla J) \sum_\mu\tilde C_\mu \frac{\p \tau^\mu}{\p t^a},
\end{equation*}
hence
\begin{equation} \label{deftCa}
\tilde C_a (\hat t) \equiv \sum_\mu \tilde C_\mu (\tau(\hat t)) \frac{\p \tau^\mu}{\p t^a}(\hat t).
\end{equation}
In particular, $\tilde C_a$ is independent of $z$.

With this understood, (\ref{tilde-Ca}) in fact is equivalent to
\begin{equation} \label{q-prod-B}
\tilde C_a = B_0 C_{a; 0} B_0^{-1}
\end{equation}
($B^{-1}_0 := (B^{-1})_0$) and the cancellation equation
\begin{equation} \label{cancel-eq}
z\p_a B = B C_a  - B_0 C_{a; 0} B_0^{-1} B,
\end{equation}
where the subscript $0$ stands for the coefficients of $z^0$
in the $z$ expansion.

Our plan is to analyze $B = B(z)$ with respect to the weight
$w := (\bar\beta, d_2) \in W$,
which carries a natural partial ordering.
The initial condition is $B_{w = (0, 0)} = {\rm Id}$:
Since we have seen that for $w = (0, 0)$, $C_a$ has only $z$ constant terms
$C_{a; (0, 0), 0}\, z^0$.
The total $z$ constant terms in (\ref{cancel-eq}) are trivially compatible.
They are $-B_0 C_{a; 0}$ on both sides.

Now we perform the induction on $W$.
Suppose that $B_{w'}$ satisfies $\T B_{w'} = B'_{w'}$ for all $w' < w$. Then
\begin{equation} \label{B-recursive}
z\p_a B_{w} = \sum_{w_1 + w_2 = w} B_{w_1} C_{a; w_2} - \sum_{w_1 + w_2 + w_3 + w_4 = w} B_{w_1, 0} C_{a; w_2, 0} B^{-1}_{w_3, 0} B_{w_4}.
\end{equation}
Write $C_{a; w} = \sum_{j = 0}^{m(w)} C_{a; w, j}\, z^j$ and $B_{w} = \sum_{j = 0}^{n(w)} B_{w, j}\, z^j$. Then (\ref{B-recursive}) implies that
$$
n(w) = \max_{w' < w} (n(w') + m(w - w')) -1.
$$
Notice that on the RHS all the $B$ terms have strictly smaller degree than $w$ except
$$
B_{w} C_{a; (0, 0)} - C_{a; (0, 0)} B_{w} + B_{w, 0} C_{a; (0, 0)} - C_{a; (0, 0)} B^{-1}_{w, 0}
$$
which has maximal $z$ degree $\le n(w)$. Moreover, by descending induction on the $z$ degree, to determine $B_{w, j}$ we need only $B_{w'}$ with $w' < w$ or $B_{w, j'}$ with $j' > j$, which are all $\T$-invariant by induction. Hence the difference satisfies
$$
\p_a (\T B_{w, j} - B'_{w, j}) = 0.
$$
The functions involved are all formal in $\bar t$ and analytic in $t^1, t^2$, and without constant term ($B_{w = (0, 0)} = {\rm Id}$). Hence $\T B_{w, j} = B'_{w, j}$.

To summarize, we have proved that for any $\hat t = \bar t + D \in H(S) \oplus \Bbb C h \oplus \Bbb C \xi$,
$$
\T B (\tau(\hat t)) \cong B' (\tau'(\hat t)).
$$
In particular, by (\ref{q-prod-B}) this implies that the $\T$-invariance of $\tilde C_a(\hat t)$. In more explicit terms, we have the $\T$-invariance of
\begin{equation} \label{3-pt-tau}
\tilde C_{a\nu}^\kappa = \sum_{n \ge 0,\, \mu} \frac{q^\beta}{n!}
\frac{\p \tau^\mu(\hat t)}{\p t^a} \langle T_\mu, T_\nu, T^\kappa,
\tau(\hat t)^n\rangle_\beta
\end{equation}
for arbitrary (basis elements) $T_\nu$, $T^\kappa \in H(X)$.

The very special case $T_\nu = 1$ leads to non-trivial invariants
only for 3-point classical invariant ($n = 0$) and $\beta = 0$, and
also $\mu = \kappa$. Since $\kappa$ is arbitrary, we have thus
proved the $\T$-invariance of $\p_a \tau$. Then
$$
\p_a(\T \tau - \tau') = \T \p_a \tau - \p_a \tau' = 0.
$$
Again since $\tau(\hat t) = \hat t$ for $(\bar\beta, d_2) = (0, 0)$, this proves
\begin{equation*}
\T \tau = \tau'.
\end{equation*}

\begin{remark}
$\tilde C_a$ is the derivative of the 2-point (Green) function at $\tau(\hat t)$:
$$
\tilde C_{a\nu}^\kappa = \frac{\p}{\p t^a} \langle\!\langle T_\nu, T^\kappa \rangle\!\rangle (\tau).
$$
\end{remark}

Now we may finish the proof of the quantum invariance (Theorem \ref{main-thm}).\smallskip

\begin{proof}
Since we have established the analytic continuation of $B$ (hence $P$) and $\tau$, by Proposition \ref{red-GMT} (reduction to special BF/GMT with $\xi$ class) and Lemma \ref{naive-ql} (naive quasi-linearity with $\xi$ class) the invariance of quantum ring is proved.
\end{proof}

\begin{remark}
We sketch an alternative shortcut to the proof to minimize the
usage of extremal functions and completely get rid of the
quasi-linearity reduction.

Indeed, by degeneration reduction (Part I \S3), the
quantum invariance problem is reduced to local models for
descendent invariants of special type. Part I Theorem~4.2
then eliminates the necessity of using $\psi$ classes and we only
need to prove the invariance of quantum ring for local models.

Now for split flops, the Birkhoff factorization matrix $B(z)$
exists uniquely.  Then quantum Leray--Hirsch theorem (Theorem
\ref{q-LH}) produces the matrix $C_a(z)$ which satisfies the
analytic continuation property. The analytic continuation of $B(z)$
is then deduced from it. In particular, (\ref{q-prod-B}) gives the
analytic continuation of $\tilde C_a(\hat t)$, namely
(\ref{3-pt-tau}), and then of $\tau(\hat t)$.

Now we apply the \emph{reduction method} used in the proof of
Proposition \ref{red-GMT}, with the role of \emph{special
insertion} $\tau_k a \xi$ being replaced by \emph{3 primary
insertions} $T_a, T_\nu, T^\kappa$ with $T_a \in H(S)$ and $T_\nu$,
$T^\kappa \in H(X)$ being arbitrary. We can do so because $\p
\tau/\p t^a = T_a + \cdots$. Notice that since $n \ge 3$, the
divisor reconstruction we need can all be performed within primary
invariants.

Namely, using Part~I Equation~(2.1) for $h$ and $\xi$, we may
reconstruct any $n \ge 3$ point primary invariants by the the one
with only two general insertions not from $H(S)$. As in Step 2 of
the proof of Part I Theorem~4.4, the moving of $\xi$ class will
always be $\T$-compatible, while the moving of $h$ class to an
insertion $t_i h^r$ may generate topological defect. The key point
is that this defect can be canceled out by the extremal quantum
corrections from some diagonal splitting term. (In fact this is the
building block of our determination of the extremal invariants in
Part~I \S2.)

This leads to a logically shorter and more conceptual proof of the
quantum invariance theorem.

We present the complete argument for at least two reasons. Firstly,
the quantum correction part (extremal case) works for non-split
flops as well. Secondly, the BF/GMT algorithm, together with the
divisorial reconstruction, provides an effective method to
determine all genus zero descendent (not just primary) invariants
for any split toric bundles.
\end{remark}

\section{Examples on quantum Leray--Hirsch} \label{MSJ-ex} 

\subsection{The toy example}
We consider the Hirzebruch surface $X=\Sigma_{-1}$ which is the $P^1$ bundle 
over $P^1$ associated to the vector bundle $\mathscr{O} \oplus \mathscr{O}(1)$.
The GW theory on $X$ can be easily determined by the classical method. 
However, we will apply quantum Leray--Hirsch to it and compare with the 
result obtained by the classical method.
 
Write $H(S) = H(P^1) = \Bbb C [p]/(p^2)$. 
By the Leray--Hirsch theorem, $H(X) = H(S)[h]/\langle h(h+p)\rangle$ has 
rank $N = 4$. 
Consider the basis $\{T_i \mid 1 \le i \le 4 \}$ given in the following order
$$
1,\, h,\,  p,\,  hp.
$$
The dual basis $\{T^i\}$ is easily seen to be given by
$$
hp,\, p,\, h + p,\,1. 
$$ 

We denote by $q = q^\ell e^{t}$, $\bar q= q^b e^{\bar t}$, where 
$b = [S] \cong [P^1]$. The Picard-Fuchs operator is
\begin{align*}
\Box_\ell &= (z\p_h)(z\p_{h+p} )- q .
\end{align*}
It leads to
\begin{align} \label{h-pf}
(z\p_h)^2 &= q-(z\p_h)(z\p_p).
\end{align}

Since $H(S) = H^0(S) \oplus H^2(S) = \Bbb C 1 \oplus \Bbb C p$ consists of the 
small parameters only, the small and big quantum rings coincide. 
It is easy to compute its QDE:
$$
z\p_p (z\p_1, z\p_p) = (z\p_1, z\p_p) \begin{pmatrix} 0 & \bar q \\ 1 & 0 \end{pmatrix}.
$$
Since $b^I = b-\ell$, we get $D_{b^I}(z) = z\p_h$. We get the lifting of the QDE to be
\begin{equation} \label{b-lift}
(z\p_p)^2 = \bar qq^{-1}\, z\p_{h}.
\end{equation}

By (\ref{h-pf}) and (\ref{b-lift}), we calculate the matrix $C_{t^a}$ of the action of $z\p_{t^a} = z\p_h$ or $z\p_p$  on $\hat T_i$ as $z\p_{t^a} \hat T_j = \sum_k C_{t^aj}^k(z) \hat T_k$ modulo $I^X$. Then 
$$
\setcounter{MaxMatrixCols}{4}
C_h =
\begin{bmatrix}
 0&q&0&-\bar q \\
1&0&0&z\bar q q^{-1}\\
0&0&0&q\\
0&-1&1&\bar q q^{-1}
\end{bmatrix},
$$ 

$$
C_p =
\begin{bmatrix}
0&0&0&\bar q\\
0&0&\bar q q^{-1}&- z\bar q q^{-1}\\
1&0&0&0\\
0&1&0&-\bar q q^{-1}
\end{bmatrix}.
$$ 
Here notice that the index $k$ (respectively $j$) corresponds to the row (respectively column) index.

We solve $B$ from $C_h$ and $C_p$ by the recursive equation (\ref{B-recursive}): $B_{2, 4} = -\bar q q^{-1}$,
$$
B =
\begin{bmatrix}
1&0&0&0\\
0&1&0& -\bar q q^{-1}\\
0&0&1&0\\
0&0&0&1
\end{bmatrix}
$$
\setcounter{MaxMatrixCols}{12}

Looking at the first column vector, it implies that in $J = PI$, one needs no Birkhoff factorization ($P(z) = 1$) and the mirror transformations reduces to identity $\tau(\hat t) = \hat t$. The full matrix system requires basis in all directions which uses the full matrix $B$ and non-trivial Birkhoff factorization is required.
$$
B = I_4   -\bar q q^{-1}e_{2, 4} , \qquad B^{-1} = I_4 +\bar q q^{-1}e_{2, 4} .
$$

From this we get $\tilde C_{t^a}$ from (\ref{q-prod-B}): $\tilde C_{t^a} = B_0 C_{t^a; 0} B_0^{-1}$, which is a minor adjustment of the matrix $C_{t^a}$.

$$
\tilde C_h =
\begin{bmatrix}
 0&q&0&0\\
1&\bar q q^{-1}&-\bar q q^{-1}&0\\
0&0&0&q\\
0&-1&1&0
\end{bmatrix},
$$ 
$$
\tilde C_p =
\begin{bmatrix}
0&0&0&\bar q\\
0&-\bar q q^{-1}&\bar q q^{-1}&0\\
1&0&0&0\\
0&1&0&0
\end{bmatrix}.
$$ 
\setcounter{MaxMatrixCols}{12}

By setting $\hat t = 0$, we get $q = q^\ell$ and $\bar q = q^b$. Thus we can read out the corresponding 3-point invariants from the above tables. For example, we look at the entries at $(2, 3)$.
\begin{equation} \label{result}
\begin{split}
\tilde C_{h3}^2 &= \langle T_2, T_3, T^2 \rangle = \langle h, p, p\rangle = -q^{-\ell}q^b, \\
\tilde C_{p3}^2 &= \langle T_3, T_3, T^2 \rangle = \langle p, p, p \rangle = q^{-\ell}q^b. 
\end{split}
\end{equation}

By classical method, we can write down the $I$-function: For $\beta= d\ell+sb$,
\begin{equation*} 
I^X_\beta = \frac{q^{d\ell} q^{sb}}{\prod\limits_{1}^{s} (p + mz)^2\prod\limits_{1}^{d} (h+ mz)\prod\limits_{1}^{d+s} (h+p + mz)} =O(z^{-2}).
\end{equation*}
This implies that $J^X=I^X$. Also we find that $I^X_\beta=O(z^{-3})$ except $s=1$, $d=-1$, and in that csae the coefficient of $z^{-2}$ is $h$. It tells us that $\langle p\rangle = q^{-\ell} q^b$ and $\langle h\rangle = -q^{-\ell} q^b$. (Here we have used $h^2=-hp$.) By the divisor axiom, $\langle h, p, p\rangle = \delta_h\delta_p\langle p\rangle =  -q^{-\ell} q^b$. Similarly, $\langle p, p, p\rangle = \delta_p \delta_p \langle p\rangle =  q^{-\ell} q^b$. These results coincide with (\ref{result}).

\begin{remark}
Notice that we state and prove the quantum Leray--Hirsch theorem (Theorem \ref{q-LH}) for certain double projective bundles (of splitting type) in order to apply it to the analytic continuation problem under flops. The same proof shows that it holds true for projective bundles, and more generally for iterated projective bundles (of splitting type).
\end{remark}

\subsection{An example with non-trivial BF/GMT}
Consider $P^1$ flop $f: X \dasharrow X'$ with bundle data 
$$(S, F, F') = (P^1, \mathscr{O} \oplus \mathscr{O}, \mathscr{O} 
\oplus \mathscr{O}(1)).$$

Write $H(S) = \Bbb C [p]/(p^2)$ with Chern polynomials
$$
f_F(h) := h^2, \qquad f_{N \oplus \mathscr{O}}(\xi) := \xi(\xi - h)(\xi - h + p).
$$
Then $H = H(X) = H(S)[h, \xi]/(f_F, f_{N \oplus \mathscr{O}})$ has dimension $N = 12$
with a basis $\{T_i \mid 0 \le i \le 11 \}$ being
$$
1,\, h,\, \xi,\, p,\, h\xi,\, hp,\, \xi^2,\, \xi p,\, h\xi^2,\, h\xi p,\, \xi^2 p,\, h\xi^2 p.
$$

Denote by $q_1 = q^\ell e^{t^1}$, $q_2 = q^\gamma e^{t^2}$, $\bar q = q^b e^{t^3}$, 
where $b = [S] \cong [P^1]$, and $\f = \f(q_1)$. The Picard--Fuchs operators are
\begin{align*}
\Box_\ell &= (z\p_h)^2 - q_1 z\p_{\xi - h}\, z\p_{\xi - h + p}, \\
\Box_\gamma &= z\p_\xi\, z\p_{\xi - h}\, z\p_{\xi - h + p} - q_2.
\end{align*}
They lead to
\begin{align} \label{h-pf-1}
(z\p_h)^2 &= \f (z\p_\xi)^2 - \f\, z\p_p\, z\p_h + \f\, z\p_p\, z\p_\xi - 2\f \, z\p_h\, z\p_\xi \\
\label{xi-pf-1} (z\p_\xi)^3 &= q_2(1 - q_1) - z\p_p (z\p_\xi)^2 + 2 z\p_h (z\p_\xi)^2 + z\p_p\, z\p_h\, z\p_\xi.
\end{align}

As before $H(S) = H^0(S) \oplus H^2(S) = \Bbb C 1 \oplus \Bbb C p$ has only small parameters with its QDE be given by
$$
z\p_p (z\p_1, z\p_p) = (z\p_1, z\p_p) \begin{pmatrix} 0 & \bar q \\ 1 & 0 \end{pmatrix}.
$$

The real difference from the previous $((0, 0), (0, -1))$ case starts with the lifting of this QDE. Now $b^I = b - \gamma$, we get $D_b = z\p_\xi \, z\p_{\xi - h}$, and the lifting becomes
\begin{equation} \label{b-lift-1}
(z\p_p)^2 = \bar qq_2^{-1}\, z\p_\xi \, z\p_{\xi - h}.
\end{equation}

By (\ref{h-pf-1}), (\ref{xi-pf-1}) and (\ref{b-lift-1}), following the steps in the proof of Theorem \ref{q-LH} we calculate $C_a$ in $z\p_a \hat T_j = \sum_k C_{aj}^k(z) \hat T_k$ modulo $I^X$. This is a lengthy yet straightforward calculation. For simplicity set $q^* = \bar qq_2^{-1}$ be the chosen admissible lift and set $\g = \f(q^*)$, $A = q_2 - q_1q_2$, $S = q_2 + q_1q_2$. We get \small
$$
\setcounter{MaxMatrixCols}{12}
C_1 =
\begin{bmatrix}
 &  &  &  & \!q_1q_2\! & \!\f\, q_2q^*\! &  &  & \!z q_1q_2\! &  &  & \\
1\! &  &  &  &  &  &  &  &  &  &  & \\
 &  &  & & &  &  &  & \!q_1q_2\! &  & & \\
 &  &  &  &  &  &  &  &  & \!q_1q_2\! &  & \!z q_1q_2 \\
 & \!-2\f\! & \!1\! &  &  & \!z\f \,q^*\! &  &  &  &  &  &  \\
 & \!-\f\! &  & \!1\! &  &  &  &  &  &  &  &  \\
 & \!\f\! &  &  &  & \!-z\f\, q^*\! &  & &  &  &  &  \\
 & \!\f\! &  &  &  &  &  &  &  &  &  & \!q_1q_2 \\
 &  &  &  &  &  & \!1\! &  &  &  &  &  \\
 &  &  &  &  & \!\f(q^* - 2)\! &  & \!1\! &  &  &  &  \\
 &  &  &  &  & \!\f(1 - q^*)\! &  &  &  &  &  &  \\
 &  &  &  &  &  &  &  &  &  & \!1\! &
\end{bmatrix},
$$ \normalsize
\small
$$
C_2 =
\begin{bmatrix}
 &  &  &  &  &  & \!A\! &  & \!z q_1q_2\! &  & \!z A \g\! & \!z^2 q_1q_2 \g \\
 &  &  &  &  &  &  &  & \!A\!  &  &  & \!z A \g \\
\!1\! &  &  &  & &  &  &  & \!2q_1q_2\! &  & \!-q_2\g\! & \!z q_1q_2 \g \\
 &  &  &  &  &  &  &  & \!q_1q_2\! &  & \! A(1  + \g)\! & \!z q_1q_2(1 + 2\g) \\
 & \!1\! &  &  &  &  &  &  &  &  & \!z^2\g\! & \!-q_2q^*(1 +\g) \\
 &  &  &  &  &  &  &  &  &  &  & \!A(1 + \g) \\
 &  & \!1\! &  &  &  &  &  &  &  & \!-z^2\g\! & \\
 &  &  & \!1\! &  &  &  &  &  &  &  & \!q_1q_2(2 + \g) \\
 &  &  &  & \!1\! &  & \!2\! &  &  &  & \!z\g\! & \!-z^2\g \\
 &  &  &  &  & \!1\! & \!1\! &  &  &  & \!2z\g\! &  \\
 &  &  &  &  &  & \!-1\! & \!1\! &  &  & \!-2z\g\! &  \\
 &  &  &  &  &  &  &  & \!-1\! & \!1\! & \!2 + \g\! & \!-2z\g
\end{bmatrix},
$$ \normalsize
and $C_3 =$ 
\small
$$
\begin{bmatrix}
 &  &  &  &  & \!-q_1q_2q^*\! &  & \!Aq^*\! &  & \!z q_1 q_2 q^*\! & \!z(q_1q_2q^* - A\g)\! & \!-z^2 q_1q_2\g \\
 &  &  &  &  &  &  &  &   & \!A q^*\! & \!A q^*\! & \!-z A\g \\
 &  &  &  & &  &  &  &  & \!q_1q_2q^*\! & \!(S - q_1q_2q^*)\g\! & \!-z q_1q_2\g \\
1\! &  &  &  &  &  &  &  &  & \!q_1q_2q^*\! & \!q_1q_2q^* - A\g\! &\!-2z q_1q_2\g \\
 &  &  & \!-q^*\! &  &  &  & \!z q^*\! &  &  & \!-z^2\g\! & \!(A + q_1q_2q^*)\g \\
 & \!1\! &  &  &  &  &  &  &  &  &  & -A\g\\
 &  &  & \!q^*\! &  &  &  & \!-z q^*\! &  &  & \!z^2\g\! & \!q_1q_2q^*\\
 &  & \!1\! &  &  &  &  &  &  &  &  & \!-q_1q_2\g \\
 &  &  &  &  & \!q^*\! &  & \!q^*\! &  & \!-z q^*\! & \!z (q^* - 2)\g\! & \!z^2\g \\
 &  &  &  & \!1\! &  &  & \!q^*\! &  &  & \!-2z \g\! &  \\
 &  &  &  &  &  & \!1\! & \!-q^*\! &  &  & \!2z\g\! &  \\
 &  &  &  &  &  &  &  & \!1\! & \!-q^*\! & \!(q^* - 2)\g\! & \!2z\g
\end{bmatrix}.
$$ \normalsize

\medskip
The appearance of $\f$ and $\g$ demonstrates the analytic dependence on the 
parameters and explains the validity of analytic continuations. 
It is now possible to solve the gauge transform $B$ inductively on 
$w = (\bar\beta, d_2)$.
The formulas are complicate and the details are thus omitted.

\begin{remark}
These examples were reported in \cite{LW}.
\end{remark}

\appendix
\section{BF/GMT and regularization}

We consider a local split $P^r$ flop $f: X \dasharrow X'$ over 
a general base $S$ and perform the BF/GMT algorithm in \S \ref{IJ-BF/GMT} 
simultaneously on $X$ and $X'$. Mysterious cancellation arisen from the 
Birkhoff factorization, which is called \emph{regularization} here, 
leads to the \emph{first step} of analytic continuation by transforming a 
rational function into its polynomial part in a canonical fashion. 
(See Proposition \ref{1step}.)

This result might lead one to believe that it is possible to prove the main
results of this paper without the quantum Leray--Hirsch theorem.
However, a closer look of the proof reveals the increasing complexity of the 
combinatorics and shows the limitation of this approach beyond the first step.
In fact, quantum Leray--Hirsch implicitly implies the existence of all higher 
order regularization. 
A direct proof along the line presented here seems, however, a rather 
non-trivial combinatorial task.

\subsection{The fundamental rational functions $Q$ and $W_{\beta_S, d_2}$}

We start by recalling some basic set-up from \S \ref{s:6}. 

Consider a local split $P^r$ flop $f: X \dasharrow X'$ with
structure data $(S, F, F')$, where $F = \bigoplus_{i = 0}^r L_i$ and 
$F' = \bigoplus_{i = 0}^r L'_i$ are sum of line bundles. 
Denote by $a_i = c_1(L_i) + h$, $b_i = c_1(L_i') + \xi - h$. 
For $\beta = \beta_S + d \ell + d_2 \gamma$, $\mu_i := L_i.\beta_S$, 
$\mu_i' := L_i'.\beta_S$. 
Thus $a_i.\beta = d + \mu_i$ and $b_i.\beta = d_2 - d + \mu_i'$. 
Also recall that $\mu^I = {\rm max}_i\, \mu_i$, 
$\mu'^I = {\rm max}_i\, \mu_i'$, 
and $\nu^I = {\rm max}\{\mu^I + \mu'^I, 0\}$. Let
\begin{equation} \label{rel-c1}
\begin{split}
\lambda_\beta &:= c_1(X/S).\beta \\
&= (c_1 + c_1').\beta_S + (r + 2) d_2 = \sum (\mu_i + \mu_i') + (r + 2) d_2,
\end{split}
\end{equation}
which depends only on $(\beta_S, d_2)$. Then the hypergeometric modification takes
the form
\begin{equation*}
I = I(t^1, t^2, \bar t, z, z^{-1}) = e^{(t^1 h + t^2 \xi)/z} \sum_{\beta
\in NE(X)} q^\beta e^{d t^1 + d_2 t^2} I^{X/S}_{\beta} J^S_{\beta_S}(\bar t)
\end{equation*}
with relative factor
\begin{equation*}
I^{X/S}_{\beta} = z^{-\lambda_\beta} \frac{\Gamma(1 +
\frac{\xi}{z})}{\Gamma(1 + \frac{\xi}{z} + d_2)} \prod_{i = 0}^r
\frac{\Gamma(1 + \frac{a_i}{z})}{\Gamma(1 + \frac{a_i}{z} + d + \mu_i)}
\frac{\Gamma(1 + \frac{b_i}{z})}{\Gamma(1 + \frac{b_i}{z} + d_2 - d + \mu_i')}.
\end{equation*}

The case $d_2 < 0$ leads to a $\xi$ factor and then $\T I_{d_2} =
I'_{d_2}$ which contains only $\T I$-effective range (by Lemma \ref{naive-ql}).
In particular the BF and GMT are all $\T$-compatible. So let $d_2 \ge 0$. 
In this case, it is then clear that the factor
${\Gamma(1 +\frac{\xi}{z})}/{\Gamma(1 + \frac{\xi}{z} + d_2)}$
contains $\xi$ except for the $\xi$-constant term $1/(d_2)!$. 
\emph{Thus this factor needs no treatment and will be ignored in the following
discussion.}
In other words, $I^{X/S}_{\beta}$ will be used as if this factor is $1$.
For the same reason (of the appearance of $\xi$ factor) that BF is needed only 
if $\lambda_\beta \le 0$.

Recall the rule for the \emph{directed product}: for any $n \in \mathbb{Z}$,
\begin{equation} \label{dir-prod}
\frac{\Gamma(1 + A)}{\Gamma(1 + A + n + x)} = \frac{1}{\prod_{j =
1}^{n} (A + j + x)} \frac{\Gamma(1 + A)}{\Gamma(1 + A + x)}.
\end{equation}

\begin{definition}
Given $(\beta_S, d_2)$, with $d_2 > -\nu^I$, the \emph{cohomology-valued fundamental rational function} $Q(\vec x)$ in $\vec x = (x_0, \ldots, x_r, y_0, \ldots, y_r)$ is defined by
\begin{equation*}
Q(\vec x) = Q_{\beta_S, d_2}(\vec x) := \prod_{i = 0}^r \frac{1}{\prod_{j =
1}^{\mu_i} (\frac{a_i}{z} + j + x_i) \prod_{j = 1}^{d_2 + \mu_i'} (\frac{b_i}{z} + j - y_i)}.
\end{equation*}
Its one variable (diagonal) version $Q(x)$ is given by setting all $x_i = x = y_i$. By abusing notations, we write $\vec x = x$ for this specialization.
\end{definition}

In terms of $Q$, with (\ref{dir-prod}) understood, the product in $I^{X/S}_\beta$ is then the specialization of 
\begin{equation} \label{QI}
Q(\vec x) \prod_{i = 0}^r \frac{\Gamma(1 + \frac{a_i}{z})}{\Gamma(1 +
\frac{a_i}{z} + x_i)} \frac{\Gamma(1 + \frac{b_i}{z})}{\Gamma(1 +
\frac{b_i}{z} - y_i)} =: {Q(\vec x) I_{\vec x\ell}}
\end{equation}
at $\vec x = d$. However, cancelations have to be understood on the RHS of (\ref{QI}) for certain $\vec x = d \in \Bbb Z$: When $x = d \in \Bbb N$, it is clear that $I_{d\ell}$
contains the factor
\begin{equation} \label{st-ord}
\frac{\Theta_{r + 1}}{z^{r + 1}} := \prod_{i = 0}^r \frac{b_i}{z}.
\end{equation}
However, for those $i$ with $d_2 - d + \mu_i' \ge 0$ (which exists when $\beta$ is $\T I$-effective), it is understood that the
factor $b_i/z$ cancels with the same term appeared in the
denominator of $Q(d)$. To make sense of the cancelation of $b_i$, we may
temporarily treat the classes $a_i$, $b_i$ as formal variables.

For those $i$ with $d + \mu_i < 0$, the
factor $a_i/z$ appears in the numerator. This is not the case for at least one $i$ (since $\beta$ is effective, or otherwise the factor $\prod_{i = 0}^r a_i = 0$ appears). Thus the leading terms
take the form
\begin{equation*}
c(d) \prod_{d + \mu_i < 0} \frac{a_i}{z} \prod_{d_2 - d + \mu_i' < 0} \frac{b_i}{z} + \cdots
\end{equation*}
in its $1/z$ expansion. The leading expression changes as $d$ varying among the integer values. This motivates the following 

\begin{definition}
Given $(\beta_S, d_2)$, a class $\beta = \beta_S + d \ell + d_2 \gamma \in NE(X)$, as well as $d$, is said to be in the \emph{unstable range} if $\beta$ is $\T I$-effective ($d \le d_2 + \mu'^I$). Otherwise it is in the \emph{stable range} ($d > d_2 + \mu'^I$).
\end{definition}

In view of (\ref{QI}) and (\ref{st-ord}), the leading $z$ order of $I^{X/S}_\beta$ which admits infinite series in $d$ is at $z^{-\lambda_\beta - (r + 1)}$. Any $z^k$ with $k > -\lambda_\beta - (r + 1)$ supports only finite number of $d$'s and all of them are within the unstable range. For this reason, we consider the shifted expression
\begin{equation} \label{W:r+1}
W[r + 1](\vec x, z, z^{-1}) := z^{r + 1} Q(\vec x)  I_{\vec x\ell}
\end{equation}
to locate the first infinite series in the $z^0$ (constant) level. 

By viewing $1/z = \Delta x_i = \Delta y_i$, $W[r + 1]$ is the \emph{multivariate extension} in multi-directions $a_i$'s and $-b_i$'s of the similar expression $W(\vec x)$ defined by setting $1/z = 0$ in $W[r + 1]$ as
\begin{equation}
W(\vec x) := z^{r + 1} \cdot \big(Q(\vec x) I_{\vec x \ell}\big)\big|_{1/z = 0}.
\end{equation}
Notice that $W(x)$ has poles at some $x = d$ if and only if non-trivial positive $z$ power survives in $W[r + 1](d)$. By our construction, $d$ must lie in the unstable range.

\begin{remark}
This extension is unique under the normalization that $I_{\vec x\ell} = 1$ at $\vec x = 0$. Indeed, $I_{x\ell}(z^{-1} = 0) = 1/\prod_{i = 0}^r \Gamma(1 + x) \Gamma(1 - x) = (\frac{\sin \pi x}{\pi x})^{r + 1}$. The naive extension gives only $1/\prod_{i = 0}^r \Gamma(1 + a_i/z + x) \Gamma(1 + b_i/z - x)$. The extra factor $\prod_{i = 0}^r \Gamma(1 + a_i/z) \Gamma(1 + b_i/z)$ is needed to recover $I_{x\ell}$.
\end{remark}

For $x = d \in \Bbb N$, applying the Taylor series for $\log (1 \pm t)$ to each $a_i$ or $b_i$ \emph{separately} and then take a product, we get
\begin{equation*}
\begin{split}
I_{d\ell} &= \prod_{i = 0}^r \frac{\prod_{j = -d + 1}^0 (\frac{b_i}{z} + j)}{\prod_{j = 1}^d (\frac{a_i}{z} +j)} \\
&=  \frac{(-1)^{(d - 1)(r + 1)} \Theta_{r + 1}}{ d^{r + 1} z^{r + 1}} \exp \sum_{k \ge 1} \frac{1}{k z^k} \Big((-1)^{k} \sum_i a_i^k H^{(k)}_d - \sum_i b_i^k H^{(k)}_{d - 1} \Big).
\end{split}
\end{equation*}
Here $H^{(k)}_d := \sum_{j = 1}^d j^{-k}$ is the $k$-th harmonic series.

Similarly in the stable range,
\begin{equation*}
\begin{split}
&Q(d) I_{d\ell} = \prod_{i = 0}^r \frac{\prod_{j = \mu_i' + d_2 - d + 1}^0 (\frac{b_i}{z} + j)}{\prod_{j = 1}^{\mu_i + d} (\frac{a_i}{z} +j)} \\
&= W_{\beta_S, d_2}(d) \frac{\Theta_{r + 1}}{z^{r + 1} } \exp \sum_{k \ge 1} \frac{1}{k z^k} \Big((-1)^{k} \sum_i a_i^k H^{(k)}_{d + \mu_i} - \sum_i b_i^k H^{(k)}_{d - d_2 - \mu_i' - 1} \Big),
\end{split}
\end{equation*}
where
\begin{equation*}
W_{\beta_S, d_2}(d) = (-1)^{\sum_{i = 0}^r (d - (d_2 + \mu_i') - 1)}  \prod_{i = 0}^r \frac{(d - (d_2 + \mu_i') - 1)!}{(d + \mu_i)!}
\end{equation*}
is the \emph{fundamental rational function} studied in \cite{Lin, Wang4}. Here for $r$ even a sign twisting $(-1)^{d}$ is understood.

For a general $d$ (say in the unstable range), the expansion in $1/z$ depends only on the \emph{length data} $d + \mu_i$ and $d_2 - d + \mu_i'$ of the curve class $\beta$. Let $I$ and $J$ be the index set with length $< 0$ and let $I^c$, $J^c$ be the complementary sets respectively. Then
\begin{equation*}
\begin{split}
Q(d) I_{d\ell} &= \frac{\prod_{i \in I} \prod_{j = \mu_i + d + 1}^0 (\frac{a_i}{z} + j)}{\prod_{i \in I^c} \prod_{j = 1}^{\mu_i + d} (\frac{a_i}{z} +j)} \frac{\prod_{i \in J} \prod_{j = \mu_i' + d_2 - d + 1}^0 (\frac{b_i}{z} + j)}{\prod_{i \in J^c} \prod_{j = 1}^{\mu_i' + d_2 - d} (\frac{b_i}{z} +j)}\\
&= (-1)^{\sum_{i \in I} \mu_i + \sum_{i \in J} (\mu_i' + d_2) + (d - 1)(|I| + |J|)} \frac{a_{I}b_{J}}{z^{|I| + |J|}} \times \\
&\qquad \frac{\prod_{i \in I} (-d - \mu_i -1)!}{\prod_{i \in I^c} (d + \mu_i)!} \frac{\prod_{i \in J} (d - d_2 - \mu_i' - 1)!}{\prod_{i \in J^c} (d_2 - d + \mu_i')!} \exp \sum_{k \ge 1} \frac{1}{k z^k} \times \\
&\qquad \qquad\Big((-1)^{k} \sum_{i \in I^c} a_i^k H^{(k)}_{d + \mu_i} + (-1)^k \sum_{i \in J^c} b_i^k H^{(k)}_{\mu_i' + d_2 - d} \\
&\qquad \qquad \qquad - \sum_{i \in I} a_i^k H^{(k)}_{-\mu_i - d - 1} - \sum_{i \in J} b_i^k H^{(k)}_{d - d_2 - \mu_i' - 1} \Big).
\end{split}
\end{equation*}

This awful looking expression is in fact very simple in nature. It is a product of $2(r + 1)$ series with each one belongs to two types, namely with negative or non-negative length data.

\subsection{Regularization of rational functions}

The key observation is that the whole situation can be considered as a product of $r + 1$ series by pairing $(L_i, L_i')$ together. As in the Calabi-Yau $P^1$ flops case (c.f.~the proof of Lemma 3.15 in \cite{Wang4}), any factor of the form (for $x$ a large integer)
\begin{equation*}
\frac{(x - \mu' - 1)!}{(x + \mu)!}
\end{equation*}
defines a rational function which has at most simple poles. (Here we take for example $\mu = \mu_i$ and $\mu' = \mu_i' + d_2$.)

Let $\mu \ge -\mu'$ (otherwise it is a polynomial and we take Taylor series), then the Laurent series at $x = d \in [-\mu, \mu']\cap \Bbb Z$ is given by
\begin{equation*}
\frac{1}{\prod_{j = -\mu}^{\mu'} (x - j)} = \frac{1}{x - d} \prod_{j \ne d; \, j = - \mu}^{\mu'} \frac{-1}{j - d}\Big(\frac{1}{1 - (x - d)/(j - d)} \Big).
\end{equation*}
Taking products over $(L_i, L_i')$ shows that the most singular term is actually the product of the simple pole from each $i$. It remains to take into account of the harmonic series and figure out the correspondences between them at poles. Substitute $x - d = \Delta x$ by $1/z$, the above expression splits at $j = d$ and becomes (again using Taylor series of $\log (1 \pm t)$)
\begin{equation*}
\frac{1}{z} \frac{(-1)^{\mu' - d}}{(\mu + d)!(-d + \mu')!} \exp \sum_{k \ge 1} \frac{1}{k z^k} \Big( (-1)^k H^{(k)}_{d + \mu} + H^{(k)}_{\mu' - d}\Big).
\end{equation*}
Notice the formal correspondence with $a_i = 1$, $b_i = -1$ up to a sign.

\medskip

The expansion of $W[r + 1]$ in $1/z$ is the Laurent expansion of $W(\vec x)$ at $\vec x = x$. The unstable range contains all possible poles of $W(x)$. The constant term at $x = d$ is the regular part ${\rm Reg}\,W(d)$.
In the stable range,
\begin{equation*}
\begin{split}
W(d) &= {\rm Reg}\,W(d) \\
&=  (-1)^{(d - 1)(r + 1)} \frac{\Theta_{r + 1}}{d^{r + 1}} \prod_{i = 0}^r \frac{1}{\prod_{j =
1}^{\mu_i} (j + d) \prod_{j = 1}^{\mu_i' + d_2} (j - d)},
\end{split}
\end{equation*}
which by definition coincides with $\Theta_{r + 1}W_{\beta_S, d_2}(d)$.

By the same process, the Taylor expansion at $x = d$ gives back to $Q(d)I_{d\ell}$ with $a_i =1$, $b_i = -1$. Notice that this does not recover $Q(d)I_{d\ell}$ completely since the process does depend on the presentation of the rational expression. Nevertheless, the above discussions lead to

\begin{lemma}
In the full range of $d$, the series expansion 
$$
z^{r + 1} Q(x)  I_{x\ell} = \sum_{k \le r + 1} W_k z^{k}
$$ 
and the Laurent expansion of $W_{\beta_S, d_2}(x)$ in $1/z$, denoted by $\sum_{k \le r + 1} w_k z^{k}$, at $x = d$ are compatible in the sense that
\begin{equation} \label{w=W}
w_k(d) = W_k(d)|_{a_i = 1, b_i = -1}.
\end{equation}
\end{lemma}

Here is a basic fact concerning polynomial parts of a rational function:

\begin{lemma} \label{polynomial}
Let $F(x)$ be a rational function with poles at $e_j$'s and
with polynomial part $P(x)$. Then
\begin{equation*}
P(e) = {\rm Reg}\,F(e) - \sum_{e_j \ne e} {\rm Pri}_{e_j} F(e).
\end{equation*}
\end{lemma}

\begin{proof}
Let $n_j = {\rm ord}_{x = e_j} F(x)$. By division and taking partial fractions, we have
\begin{equation*}
F(x) = P(x) + \frac{R(x)}{\prod_{j} (x - e_j)^{n_j}} = P(x) +
\sum_{j} {\rm Pri}_{e_j} F(x).
\end{equation*}
If $e \not\in \{e_j\}$, then ${\rm Reg}\,F(e) = F(e)$ and the
lemma holds. If $e = e_i$ for some $i$, then
\begin{equation*}
F(x) = {\rm Pri}_{e} F(x) + \Big( P(x) + \sum_{j \ne
i} {\rm Pri}_{e_j} F(x) \Big)
\end{equation*}
and the lemma again holds.
\end{proof}

Combining both lemmas leads to results on the first stable series $W_0(d)$. For ease of notations, denote by 
$$
A = A(q, z) = z^{-\lambda_\beta - (r + 1)} q^{\beta_S} q^{d_2 \gamma}
$$ 
the basic factor centered at the first stable series ($\lambda_\beta \equiv c_1(X/S).\beta$).

\begin{proposition} \label{1step}
Given $(\beta_S, d_2)$ with $c_1(X/S).\beta \le -(r + 1)$, so that the first stable series is located at non-negative $z$ degree, the ``partial Birkhoff factorization'' up to the first stable series 
$$
P_1(z) I :=  I - A \sum_{r + 1 \ge k \ge 1;\, e} z^{k} q^e \widehat{W}_k(e) I 
$$
leads to polynomial values $P_{\beta_S, d_2}(d) q^d$ at order $z^{-c_1(X/S).\beta - (r + 1)}$ in the stable range. This also holds for general $d$ if we consider $\T P_1(z) I^X - P_1'(z) I^{X'}$. In particular, this leads to analytic continuations of $P_1(z) I$ up to $z^{-c_1(X/S).\beta - (r + 1)}$.

The compatibility of partial BF operators $\T P_1(z) = P_1'(z)$ always holds even for $c_1(X/S).\beta > -(r + 1)$. In that case $\T P_1(z) I^X - P_1'(z) I^{X'} = 0$ for all non-negative $z$ degree terms lying over $(\beta_S, d_2)$.
\end{proposition}

\begin{proof}
For $1 \le k \le r + 1$, a target term \emph{with an additional $z^k$ power} lies in $A W_k z^{k} q^{e\ell}$ and takes the form 
$$
A c_{IJ}^k(e) a_I b_J z^{k} q^{e\ell}
$$ 
with $|I| + |J| = r + 1 - k \le r$. In particular there is a corresponding $\T$-compatible term on the $X'$ side given by $(\T A) c_{IJ}^k(e) b'_I a'_J z^{k} q^{-e\ell'}$.

For a divisor $D$, the naive quantization has the effect $\hat D = z\p_D = D + z\delta_D$, where $\delta_D$ is the number operator which acts on $q^\beta$ by $\delta_D\, q^\beta = (D.\beta)\, q^\beta$. Then in the partial BF procedure (c.f.~Theorem \ref{BF/GMT})
\begin{equation*}
I - A \sum_{k, e, I, J} c_{IJ}^k(e) z^{k} q^e \prod_{i \in I}(a_i + z\delta_{a_i}) \prod_{j \in J} (b_j + z\delta_{b_j}) I,
\end{equation*}
the first term $a_I b_J$ in the product cancels the target term. 

Modulo higher $\beta_S$ and $d_2\gamma$, we only need to consider extremal contribution $\sum_{d\ge 1} I_{d\ell}\,q^d$ to the product ($q := q^\ell$). The highest $z$ degree comes from
\begin{equation*}
\begin{split}
&-A c_{IJ}^k(e) z^{k + (r + 1 - k)} q^e \prod \delta_{a_i} \prod \delta_{b_j} \sum_{d \ge 1} \frac{(-1)^{(d - 1)(r + 1)} \Theta_{r + 1}}{d^{r + 1} z^{r + 1}} q^d \\
&\quad = -(-1)^{|J|} A c_{IJ}^k(e) \Theta_{r + 1} \sum_{d \ge 1} \frac{(-1)^{(d - 1)(r + 1)}}{d^k} q^{d + e} \\
&\quad = -(-1)^{|J|} A c_{IJ}^k(e) \Theta_{r + 1} \sum_{d \ge e + 1} \frac{(-1)^{(d - e - 1)(r + 1)}}{(d - e)^k} q^{d}.
\end{split}
\end{equation*}

By construction, we have for each fixed $k$ and unstable $e$ that
\begin{equation}
\sum_{|I| + |J| = r + 1 - k} (-1)^{|J|} c^k_{IJ}(e) = W_k(e)|_{a_i = 1, b_i = -1} = w_k(e).
\end{equation}

If $d$ is in the stable range, then summing all the unstable terms with positive $z$ power gives rise to the principal part of $W_{\beta_S, d_2}(d)$. Thus the result follows by a careful check on the signs. (Namely $\Bbb Z_2$ graded if $r$ is even.)

If $d$ is in the unstable range, then there are two places in the proof of polynomiality which need to be modified.

Firstly, $W_0(d)$ is related to ${\rm Reg}\,W_{\beta_S, d_2}(d)$ if we set $a_1 = 1$, $b_i = -1$. Alternatively, as $d$ makes sense on both $X$ and $X'$ sides, we have also the relation on topological defect
\begin{equation} \label{top-defect}
\T W_0(d) - W'_0(d) = (-1)^{r + 1} {\rm Reg}\, W_{\beta_S, d_2}(d) \Theta'_{r + 1},
\end{equation}
where $\Theta'_{r + 1} = \prod_{i = 0}^r b_i' = \prod_{i = 0}^r (c_1(L_i) + \xi' - h')$. 
(This follows from Part I. Indeed it is clear that the difference is a scalar multiple of $\Theta'_{r + 1}$ since it is in the kernel of the multiplication map by $\xi'$.)

Secondly, the shifting of $k$-th order pole by $e$ only works for those $e < d$. Those poles at $e$ with $e > d$ are missing from the formula on the $X$ side. Thus to receive a complete correction of the principal part from all $e \ne d$ we need (and only need) to consider $\T P_1(z) I^{X/S} -P_1'(z) I^{X'/S}$. 

For the last statement, notice that $\f(q) + \f(q^{-1}) = (-1)^r$ is formally equivalent to the vanishing of the Euler series $E(q) := \sum_{d \in \Bbb Z} q^d = 0$. Hence
$$
\sum_{d \in \Bbb Z} P_{\beta_S, d_2}(d) q^d = P_{\beta_S, d_2}(qd/dq) E(q) = 0.
$$

The proof is complete.
\end{proof}

\subsection{A remark on higher regularization}

Next we move to the Birkhoff factorization up to the second stable series. This step is needed only if 
$$
-c_1(X/S).\beta -(r + 1) \ge 1.
$$ 
Harmonic series appears naturally and the expected regularization into polynomials becomes much more tricky. An simple useful fact is that the difference of two harmonic series is a rational function.

Let $\lambda_j = c_1(L_j)$ and $\lambda_j' = c_1(L_j')$. Denote by $e$ an index in the unstable range, then the partial BF with one more order reads as
\begin{equation} \label{BF-2}
\begin{split}
P_2(z) I := I &- A \sum_{r + 1 \ge k \ge 1;\, e} z^{k} q^e \widehat{W}_k(e) I \\
& - A \sum_{d: {\rm stable}} q^d P_{\beta_S, d_2}(d) \widehat\Theta_{r + 1} I \\
& - A \sum_{d: {\rm unstable}} q^d  \Big(\widehat W_0(d) - \sum_{e < d} {\rm Pri}_{e}(d) \widehat\Theta_{r + 1} \Big) I
\end{split}
\end{equation}
where $\Theta_{r + 1} = \prod_{j = 0}^r b_j = \prod_{j = 0}^r ((\lambda_j + \lambda'_j) + \xi - a_j)$ and
\begin{equation*}
\begin{split}
\widehat \Theta_{r + 1} := \prod_{j = 0}^r z\p_{b_j} - (-1)^{r + 1}\prod_{j = 0}^r z\p_{a_j}
\end{split}
\end{equation*}
(since $\prod a_j = 0$, the corresponding quantization product is removed). By the construction, the first stable series vanishes automatically. 

Now we investigate the second stable series, namely the 
$$
A z^{-1} = z^{-\lambda_\beta - (r + 1) - 1} q^{\beta_S} q^{d_2 \gamma}
$$ 
degree terms. They all contain the factor $(-1)^{(d - 1)(r + 1)}\Theta_{r + 1}$ hence we may remove $\xi$ from the remaining classes. 

The main terms come from the first two series in (\ref{BF-2}). The terms from $I$ are degree $A$ terms multiplied by the following harmonic series
\begin{equation*}
\begin{split}
&-\sum a_i H_{d + \mu_i} - \sum b_i H_{d - 1 - \mu'_i - d_2} \\
&\qquad = h \sum (-H_{d + \mu_i} + H_{d - 1 - \mu'_i - d_2}) - \sum (\lambda_i + \lambda'_i) H_{d - 1} \\
&\qquad \qquad + \sum \lambda_i (H_{d} - H_{d + \mu_i}) + \sum \lambda'_i (H_{d - 1} - H_{d - 1 - \mu'_i - d_2}) - \sum \lambda_i /d.
\end{split}
\end{equation*}

The terms from the second series form a sum over $k, e$, which has two parts: One with $(z\delta_h)^r$ on the second extremal series, which is 
$$
\sum (-1)^{|J|} c^k_{IJ}(e) A z^{-1} \Theta_{r + 1}
$$ 
multiplied by 
$$
-\sum a_i H_d - \sum b_i H_{d - 1} = -(r + 1)h/d -\sum (\lambda_i + \lambda'_i) H_{d - 1} - \sum \lambda_i/d,
$$ 
and another one with \emph{one less} differentiation $(z\delta_h)^{r - 1}$ on the top extremal term, which receives a factor
\begin{equation*}
(\sum_{i \in I} a_i - \sum_{i \in J} b_i)/d  = (r + 1 - k)h/d + \sum_{i \in I} \lambda_i/d - \sum_{i \in J} \lambda'_i/d.
\end{equation*}
For each $(k, e)$, we find correction factor 
$$
-\frac{kh}{d} \quad \Big(\mapsto -\frac{kh}{d - e} \quad \mbox{after shifted by $q^e$}\Big),
$$
hence it gives rise to derivative of $(d - e)^{-k}$. 

In the stable range, the first corresponding terms then lead to derivative, denoted by ${}^\bullet$ here, of the rational function. Since $f^\bullet - \sum g^\bullet = (f - \sum g)^\bullet$, they combine to the polynomial 
\begin{equation}
h P^\bullet_{\beta_S. d_2}(d),
\end{equation}
which is expected for the purpose of analytic continuations.

Similarly, by shifting $H_{d - e - 1}$ to $H_{d - 1}$ which is only up to a rational function in $d$, the second corresponding terms combine to
\begin{equation} \label{BAD}
-(c_1 + c'_1) P_{\beta_S, d_2}(d) H_{d - 1}.
\end{equation}
This is unfortunately the \emph{trouble term}, due to the appearance of $H_{d - 1}$. 

Finally, the last terms combine to 
$$
-c_1 P_{\beta_S, d_2}(d)/d.
$$

For unstable range, as in the proof of Proposition \ref{1step}, it is expected that similar calculation holds if we consider $\T P_2(z) I^{X/S} - P'_2(z) I^{X'/S}$. 

Combining the \emph{third series} in (\ref{BF-2}) and the one on the $X'$ side does produce correction terms, via \emph{harmonic convolution}, to cancel out the bad term (\ref{BAD}). The actual calculation is however getting more and more involved.

Simple examples such that the higher regularization is explicitly carried out can be found in \cite{LLW2}. But the elementary method used there (harmonic convolution etc.) does not seem to apply to the general case. This was one of the major motivations for us to develop the quantum Leray--Hirsch theorem during the early stage of this project after \cite{LLW}.

\end{document}